\documentclass[english,12pt]{smfart}
\usepackage{amssymb}
\usepackage{amsfonts}
\usepackage{amscd}
\usepackage[T1]{fontenc}
\usepackage{color}
\usepackage{tikz}

\newbox\removebox
\newcommand\remove[1]{%
\setbox\removebox=\ifmmode\hbox{$#1$}\else\hbox{#1}\fi%
\leavevmode
\rlap{\textcolor{blue}{\vrule height0.8ex depth-0.6ex width\wd\removebox}}%
\box\removebox
}
\long\def\bigremove#1{%
\par\setbox\removebox=\vbox{#1}%
\vbox{%
\vbox to0pt{\hbox{\tikz\draw[color=blue,thick] (0,0) -- (\wd\removebox,-\ht\removebox)  (\wd\removebox,0) -- (0,-\ht\removebox);}}
\box\removebox
}
}

\usepackage{mathrsfs} 

\def\VF{\mathrm{VF}}
\def\RF{\mathrm{RF}}
\def\VG{\mathrm{VG}}

\def\deg{\operatorname{deg}}
\def\ac{{\overline{\rm ac}}}

\def\longhookrightarrow{\mathrel\lhook\joinrel\longrightarrow}

\def\11{{\mathbf 1}}
\def\AA{{\mathbb A}}

\def\CC{{\mathbb C}}

\def\FF{{\mathbb F}}

\def\NN{{\mathbb N}}

\def\PP{{\mathbb P}}
\def\QQ{{\mathbb Q}}
\def\RR{{\mathbb R}}

\def\ZZ{{\mathbb Z}}

\def\cA{{\mathcal A}}
\def\cB{{\mathcal B}}
\def\cC{{\mathscr C}}

\def\cL{{\mathcal L}}
\def\cM{{\mathcal M}}

\def\cO{{\mathcal O}}

\def\cT{{\mathcal T}}

\def\cX{{\mathcal X}}

\def\llp{\mathopen{(\!(}}
\def\llb{\mathopen{[\![}}

\def\rrp{\mathopen{)\!)}}
\def\rrb{\mathopen{]\!]}}

\newtheorem{thm}[subsubsection]{Theorem}
\newtheorem{lem}[subsubsection]{Lemma}

\newtheorem{prop}[subsubsection]{Proposition}

\newtheorem{introthm}{Theorem}

\theoremstyle{definition}
\newtheorem{defn}[subsubsection]{Definition}

\newtheorem{def-prop}[subsubsection]{Proposition-Definition}
\newtheorem{def-theorem}[subsubsection]{Theorem-Definition}
\newtheorem{def-lem}[subsubsection]{Lemma-Definition}
\newtheorem{notation}[subsubsection]{Notation}
\newtheorem{setting}[subsubsection]{Setting}

\theoremstyle{remark}
\newtheorem{remark}[subsubsection]{Remark}

\theoremstyle{plain}

\numberwithin{equation}{subsection}

\renewcommand{\theequation}{\thesubsection.\arabic{equation}}

  {\par\medskip\noindent #1\par\begingroup%
    \advance\leftskip by 1em\advance\rightskip by 1em}%
  {\par\endgroup}

\newcommand{\ord}{\operatorname{ord}}

\newcommand{\abs}[1]{{\lvert#1\rvert}}

\newcommand{\ceil}[1]{\lceil#1\rceil}
\newcommand{\HF}{\operatorname{HF}}
\newcommand{\HP}{\operatorname{HP}}
\newcommand{\LT}{\operatorname{LT}}
\newcommand{\set}[1]{\left\{#1\right\}}
\newcommand{\as}{{\mathrm as}}
\newcommand{\alg}{{\mathrm{alg}}}

\begin{document}

\setcounter{tocdepth}{1} 

\title[Uniform Yomdin-Gromov parametrizations]{Uniform Yomdin-Gromov parametrizations and points of bounded height in valued fields}


\author{Raf Cluckers}
\address{Universit\'e de Lille, Laboratoire Painlev\'e, CNRS - UMR 8524, Cit\'e Scientifique, 59655
Villeneuve d'Ascq Cedex, France, and,
KU Leuven, Department of Mathematics,
Celestijnenlaan 200B, B-3001 Leu\-ven, Bel\-gium}
\email{Raf.Cluckers@univ-lille.fr}
\urladdr{http://rcluckers.perso.math.cnrs.fr/}

\author{Arthur Forey}
\address{ETH Z\" urich, D-Math, R\" amistrasse 101, 8092 Z\" urich, Switzerland}
\email{arthur.forey@math.ethz.ch}
\urladdr{https://people.math.ethz.ch/$\sim$aforey/}

\author{Fran\c cois Loeser}
\address{Institut universitaire de France, Sorbonne Universit\'e,  UMR 7586 CNRS, Institut Math\'ematique de Jussieu, F-75005 Paris, France}
\email{Francois.Loeser@imj-prg.fr}
\urladdr{https://webusers.imj-prg.fr/$\sim$francois.loeser/}

\begin{abstract}
We prove a uniform version of non-Archimedean Yomdin-Gromov parametrizations in a definable context with algebraic Skolem functions in the residue field.  The parametrization result allows us to bound the number of $\FF_q[t]$-points of bounded degrees of algebraic varieties, uniformly in the cardinality $q$ of the finite field $\FF_q$ and the degree, generalizing work by Sedunova for fixed $q$. We also deduce a uniform non-Archimedean Pila-Wilkie theorem, generalizing work by Cluckers-Comte-Loeser.
\end{abstract}

\maketitle

\section{Introduction}\label{intro}

Since the pioneering work \cite{bombieri_pila89}, the 
determinant method of Bombieri and Pila has been used in various contexts to count integer and rational points of bounded height in algebraic or analytic varieties.
Parametrization results, as initiated by  Yomdin and Gromov, play a prominent role in some of the most fruitful applications of this method, such as  the Pila and Wilkie counting theorem for definable sets in $o$-minimal structures \cite{PiWi}. In the non-Archimedean setting, Cluckers, Comte and Loeser prove in \cite{CCL-PW} an analogue of the Pila-Wilkie counting theorem, but for subanalytic sets in $\QQ_p$, the field of $p$-adic numbers. Their proof relies also on a Yomdin-Gromov type parametrization result. The aim of this paper is to extend their result to obtain  bounds uniform in $p$ for some counting points of bounded height problems,
over $\QQ_p$ and over $\FF_p\llp t \rrp$. Before discussing our parametrization result, we shall start by presenting the applications to point counting.

\subsection{Point counting in function fields}

For $q$ a prime power, consider the finite field with $q$ elements $\FF_q$ and for each positive integer $n$, let $\FF_q[t]_n$ be the set of polynomials with coefficients in $\FF_q$ and degree (strictly) less than $n$. Cilleruelo and Shparlinski \cite{CS_concentration_points} have raised the question to bound the number of $\FF_q[t]_n$-points in plane curves. That question was settled by Sedunova \cite{sedunova_BP}. A particular case of our main theorem is a uniform version of her results. We refer to Theorem \ref{thm:main:uniform:countingFq} for a more general statement, namely for $X$ of arbitrary dimension.
For an affine variety $X$ defined over a subring of  $\FF_q\llp t\rrp$, write $X(\FF_q[t])_n$ for the subset of $X(\FF_q\llp t\rrp)$ consisting of points whose coordinates lie in $\FF_q[t]_n$.

\begin{introthm}
\label{thmintro_points-curves}
Fix an integer $\delta>0$. Then there exist real numbers $C=C(\delta)$ and $N=N(\delta)$ such that for each prime $p>N$, each $q=p^\alpha$, each integer $n>0$ and each irreducible plane curve $X\subseteq \AA^2_{\FF_q\llp t\rrp}$ of degree $\delta$
one has
\[
\# X(\FF_q[t])_n \leq C n^2 q^{\ceil{\frac{n}{\delta}}}.
\]
\end{introthm}

A similar statement is proved by Sedunova \cite{sedunova_BP}, for fixed $q$. More precisely, she proves that fixing $\delta$, $q$ and $\epsilon>0$, there exist a constant $C'=C'(\delta,q,\varepsilon)$ such that for each $X\subseteq \AA^2_{\FF_q[t]}$ irreducible plane curve of degree $\delta$ and positive integer $n$,
\[
\# X(\FF_q[t])_n \leq C'  q^{n(\frac{1}{\delta}+\varepsilon)}.
\]
Observe that our result improves Sedunova's one by replacing the $\varepsilon$ factor by a  polylogarithmic term.
By the very nature of our  methods, which are model-theoretic, we are however unable to establish such a result for $q$ a power of a small prime $p$.

\subsection{A uniform non-Archimedean point counting theorem}

We state a uniform version of the Cluckers-Comte-Loeser non-Archimedean point counting theorem. A semi-algebraic set is a set defined by a first order formula in the language $\cL_\mathrm{div}=\set{0,1,+,\cdot,\mid}$ and parameters in $\ZZ\llb t\rrb$, where $\mid$ is a relation interpreted by $x\mid y$ if and only if $\ord(y)\leq \ord(x)$, with $\ord$ the valu\-ation. As usual, we will identify definable sets with the
formulas that define them. Subanalytic sets  are definable sets in the language obtained by adding a new symbol for each analytic function with coefficients in $\ZZ\llb t\rrb$
to the language $\cL_\mathrm{div}$. For each local field $L$ of characteristic zero, we fix a choice of uniformizer $\varpi_L$ and view it as a $\ZZ\llb t \rrb$-ring by sending $t$ to $\varphi_L$. Hence, we can consider the $L$-points of a semi-algebraic or subanalytic set, for $L$ a local field of any characteristic. The notion of semi-algebraic and subanalytic sets considered in Section \ref{section_PiWi} is slightly more general than the one considered here, see also Setting \ref{setting-T_DP}.

The dimension of a subanalytic set $X$ is the largest $d$ such that there exists a coordinate projection $p$ to a linear space of dimension $d$ such that $p(X)$ contains an open ball. A subanalytic set is said to be of pure dimension $d$ if for each $x\in X$ and every ball $B$ centered at $x$, $X\cap B$ is of dimension $d$. If $X\subseteq L^n$, we denote by $X^\alg$ the union of all semi-algebraic curves of pure dimension 1 contained in $X$. Observe that in general, $X^\alg$ is not semi-algebraic (nor subanalytic).

If $X\subseteq K^m$ and $H\geq 1$, with $K$ a field of characteristic zero, we denote by $X(\QQ,H)$ the set of $x=(x_1,\dots,x_m)\in X\cap \QQ^m$ that can be written as $x_i=a_i/b_i$, with $a_i,b_i\in \ZZ$, $\abs{a_i},\abs{b_i}\leq H$ (where $\abs{\cdot}$ is the Archimedean absolute value). If $X\subseteq L^m$, where $L=\FF_q\llp t\rrp$, we denote by $X(\FF_q(t),H)$ the set of $x=(x_1,\dots,x_m)\in X\cap \FF_q(t)^m$ that can be written as $x_i=a_i/b_i$, with $a_i,b_i\in \FF_q[t]$ of degree less or equal than $\log_q(H)$.

The following result is a particular case of Theorem \ref{thm:unifCCL-PW}. It provides a uniform version of Theorem 4.2.4 of \cite{CCL-PW}.
\begin{introthm}
\label{thmintro-PiWi}
Let $X$ be a subanalytic set of dimension $m$ in $n$ variables, with $m<n$. Fix $\varepsilon>0$. Then there exists a $C=C(X,\varepsilon)$, $N=N(X,\varepsilon), \alpha=\alpha(n,m)$ and a semi-algebraic set $W^\varepsilon \subseteq X$ such that for each $H\geq 1$ and each local field $L$, with residue field of characteristic $p_L>N$ and cardinal $q_L$, the following holds. We have $W^\varepsilon(L) \subseteq X(L)^\alg$ and if $L$ is of characteristic zero,
\[
\# (X\backslash W^\varepsilon)(L)(\QQ,H)\leq C(X,\varepsilon)q_L^\alpha H^\varepsilon.
\]
If $L$ is of positive characteristic, then
\[
\# (X\backslash W^\varepsilon)(L)(\FF_{q_L}(t),H)\leq C(X,\varepsilon)q_L^\alpha H^\varepsilon.
\]
\end{introthm}
An important step toward the proof of Theorem \ref{thmintro-PiWi} is Proposition \ref{prop-covering-bhp-hypersurfaces}, which states that integer points of height at most $H$ and lying in a subanalytic set are contained in an algebraic hypersurface of  a degree which depends polylogarithmically in $H$.

\subsection{Uniform Yomdin-Gromov parametrizations}
The proofs of Theorems \ref{thmintro_points-curves} and \ref{thmintro-PiWi} rely on the following parametrization result.

Fix a positive integer $r$. Let $L$ be a local field, or more generally a valued field endowed with its ultrametric absolute value $\abs{\cdot}$. A function $f : U\subseteq L^m \to L$ is said to satisfy $T_r$-approximation if for each $y\in U$ there is a polynomial $T^{<r}_{f,y}(x)$ of degree less than $r$ and coefficients in $L$ such that for each $x,y\in U$,
\[
\abs{f(x)-T^{<r}_{f,y}(x)}\leq \abs{x-y}^r.
\]
A $T_r$-parametrization of a set $X\subseteq L^n$ is a finite partition of $X$ into pieces $(X_i)_{i\in I}$ and for each $i\in I$, a subset $U_i\subseteq \cO_L^m$ and a surjective function $f_i : U_i\to X_i$ that satisfies $T_r$-approximation.

The following statement is a particular case of Theorem \ref{uniform-Tr-approx}.
\begin{introthm}
\label{thmintro-unif-YG-para}
Let $X$ be a subanalytic set included in some cartesian power of the valuation ring, and of dimension $d$. Then there exist integers $C$ and $N$ such that if $L$ is a local field of residue characteristic $p_L\geq N$, then for each integer $r>0$, there is a partition of $X(L)$ into $Cr^d$ pieces such that for each piece $X_i$,
there is a surjective function $f_i : U_i\subseteq \cO_L^d\to X_i$ satisfying $T_r$-approximation on $U_i$.
\end{introthm}
Observe that in the preceding theorem, we do not claim that the $X_i$ and $f_i$ are subanalytic, and indeed they are not in general.

Theorem \ref{thmintro-unif-YG-para} is used to deduce Theorems \ref{thmintro_points-curves} and \ref{thmintro-PiWi}, using an analog of the Bombieri-Pila determinant method. To be more precise, we follow closely the approach by Marmon \cite{marmon_generalization_2010} in order to prove Theorem \ref{thmintro_points-curves}.

Note also that from Theorem 3.1.3 of \cite{CCL-PW}, we can deduce by compactness a result similar to Theorem \ref{thmintro-unif-YG-para} but for fixed $r$ and with the number of pieces depending polynomialy in the cardinal of the residue field. Such a result is however too weak to obtain a non-trivial bound in Theorem \ref{thmintro_points-curves}.

The way we make Theorem \ref{thmintro-unif-YG-para} independent of the residue field is by adding algebraic Skolem functions in the residue field to the language. This enables us to work in a theory where the model-theoretic algebraic closure is equal to the definable closure. The functions involved in the parametrization are definable in such an extension of the language. Theorem \ref{thmintro-unif-YG-para} is then deduced from a $T_1$-parametrization theorem \ref{strong-T1-approx}, where the functions are required to satisfy an extra technical condition call condition $(*)$, see Definition \ref{def-condition-star}. Such a condition implies that the function (when interpreted in any local field of large enough residue characteristic) is analytic on any box contained in its domain. This allows us to deduce the $T_r$-parametrization result by precomposing with power functions.

A first step toward Theorem \ref{thmintro-unif-YG-para} is Theorem \ref{Lip}, which states that the domain of a definable (in the above sense) function that is locally 1-Lipschitz can be partitioned into finitely many definable pieces on which the function is globally 1-Lipschitz. It is similar to Theorem 2.1.7 of \cite{CCL-PW}, but there the domain is partitioned into infinitely many pieces parametrized (definably) by the residue field. The improvement is made possible by the fact that we work in a theory with algebraic Skolem functions in the residue field.

Let us finally observe that the number of pieces of the $T_r$-parametrization is $Cr^d$, where $d$ is the dimension. In the Archimedean setting, a similar result has recently been proven by Cluckers, Pila and Wilkie \cite{CPW_UYGpara}, but there the number of pieces of the $T_r$-parametrization is a polynomial in $r$ of non-explicit degree in general; in the case of $\RR_{\mathrm an}$, this degree in $r$ has meanwhile been made explicit in Theorem 2 of \cite{BN3}.

The paper is organized as follows. Section \ref{sec-GLC} is devoted to the fact that one can go from local to global Lipschitz continuity. In Section \ref{section_analytic_parametrizations}, we prove our main parametrization result. Sections \ref{sec:Fqt-count} and \ref{section_PiWi} are devoted  to applications, the first to the counting of points of bounded degree in $\FF_q[t]$, the second to the uniform non-Archimedean  Pila-Wilkie theorem.

\subsection{Acknowledgements}

The authors would like to thank I.~Halupczok for sharing inspiring ideas towards the piecewise Lipschitz continuity results of this paper.
We thank also  Z.~Chatzidakis and M.~Hils for useful discussions and comments.
R.C. was partially supported by the European Research Council under the European Community's Seventh Framework Programme (FP7/2007-2013) with ERC Grant Agreement nr. 615722
MOTMELSUM, by the Labex CEMPI  (ANR-11-LABX-0007-01), and by KU Leuven IF C14/17/083. A.F. was partially supported  by ANR-15-CE40-0008 (D\'efig\'eo) and by DFG-SNF lead agency program grant number 200020L\_175755.
F.L. was partially supported  by ANR-15-CE40-0008 (D\'efig\'eo) and by the Institut Universitaire de France.

\section{Global Lipschitz continuity}
\label{sec-GLC}

For $h:D\subseteq A\times B\to C$ any function between sets and for $a\in A$, write $D_a$ for the set $\{b\in B\mid (a,b)\in D\}$ and write $h(a,\cdot)$ or $h_a$ for the function which sends $b\in D_a$ to $h(a,b)$. We use similar notation $D_a$ and $h(a,\cdot)$ or $h_a$ when $D$ is a Cartesian product $\prod_{i=1}^n A_i$ and $a\in p(D)$ for some coordinate projection $p:D\to \prod_{i\in I\subseteq \{1,\cdots, n\}} A_i$.

\subsection{Tame theories}

We consider tame structures in the sense of \cite{CCL-PW}, Section 2.1. We recall their definition here.

Let $\cL_{\rm Basic}$ be the first order language with the sorts $\VF$, $\RF$ and $\VG$, and symbols for addition and a constant 0 on $\VF$, for $\ac:\VF\to \RF$, $|\cdot|:\VF\to\VG$, and for the order and the multiplication, and constant 0 on $\VG$, and a constant 0 on $\RF$.
Let $\cL$ be any expansion of $\cL_{\rm Basic}$. 
By $\cL$-definable we mean $\emptyset$-definable in the language $\cL$, and likewise for other languages than $\cL$.
By contrast, we will use the word `definable' more flexibly in this paper and it may involve parameters from a structure.
Write $\VF^0=\{0\}$, $\RF^0=\{0\}$, and $\VG^0=\{0\}$, with a slight abuse of notation.
Note that $\cL$ may have more sorts than $\cL_{\rm Basic}$, since it is an arbitrary expansion.

We assume that all the $\cL$-structures we consider are models of $\cT_{\rm Basic}$, the $\cL_{\rm Basic}$-theory stating that $\VF$ is an abelian group, that $\VG=\VG^\times\cup \set{0}$, with $\VG^\times$ a (multiplicatively written) ordered abelian group, $\abs{\cdot} : \VF\to \VG$ a surjective ultrametric  absolute value (for groups), $\ac : \VF\to \RF$ is surjective and $\ac^{-1}(0)=\set{0}$.

Consider an $\cL$-structure with $K$ for the universe of the sort $\VF$, $k$ for $\RF$, and $\Gamma$ for $\VG$. We usually denote this structure by $(K,\cL)$.
\begin{remark}
Most often, $K$ will be a valued field, $k$ its residue field and $\Gamma$ its value group (hence the sort names $\VF$, $\RF$ and $\VG$), although here we just require $K$ to be a (valued) abelian group.
\end{remark}

We define an open (resp. closed) ball as a subset $B\subseteq K$ of the form
$\set{x\in K\mid \abs{x-a}< \alpha}$ (resp. $\set{x\in K\mid \abs{x-a}\leq \alpha}$, for some $a\in K$ and $\alpha\in \Gamma^\times$.

We define $k^\times$ as $k\backslash\set{0}$. For $\xi\in k$ and $\alpha\in \Gamma$, we introduce the notation
\[
A_{\xi,\gamma}=\set{x\in K\mid \ac(x)=\xi,\abs{x}=\alpha}.
\]
Observe that if $\xi\in k^\times$ and $\alpha\in \Gamma^\times$, then $A_{\xi,\gamma}$ is an open ball.


We put on $K$ the valuation topology, that is, the topology with the collection of open balls as base and the product topology on Cartesian powers of $K$. 

For a tuple $x=(x_1, \cdots, x_n)\in K^n$, set $|x|=\max_{1\leq 1\leq n} \set{|x_i|}$.

\begin{defn}\label{def:lip}
Let $f: X\subseteq K^m \to K$ be a function. The function $f$ is called $1$-Lipschitz continuous (globally on $X$) or, in a short form, $1$-Lipschitz if for all $x$ and
$y$ in $X$,
\[
\abs{f(x) -f(y)} \leq \abs{x-y}.
\]
The function $f$ is called locally $1$-Lipschitz
if, locally around each point of $X$, the function $f$ is $1$-Lipschitz continuous.
\end{defn}

For $\gamma\in \Gamma^\times$, a function $f: X\subseteq K^n \to K$
is called $\gamma$-Lipschitz if for all $x$ and
$y$ in $X$,
\[
\abs{f(x) -f(y)} \leq \gamma\cdot \abs{x - y}.
\]

\begin{defn}[s-continuity]\label{defjacprop}
Let $F:A\to K$ be a function for some set $A\subseteq K$. We say that $F$
is s-continuous if for each open ball $B\subseteq A$ the set $F(B)$ is either a singleton or an open ball, and
there exists $\gamma= \gamma (B)\in \Gamma$ such that
\begin{equation}\label{s-gamma}
\abs{F(x)-F(y)} = \gamma \abs{x-y}\ \mbox{ for all $x,y\in B$}.
\end{equation}
\end{defn}

If a function $g:U\subseteq K^n\to K$ on an open $U$ is $s$-continuous in, say, the variable $x_n$, by which we mean that $g(a,\cdot)$ is $s$-continuous for each choice of $a=(x_1,\ldots,x_{n-1})$ then we write $|\partial g/\partial x_n ( a,x_n)|$ for the element $\gamma\in \Gamma$ witnessing the s-continuity of $g(a,\cdot)$ locally at $x_n$, namely, $\gamma$ is as in (\ref{s-gamma}) for the function $F(\cdot)=g(a,\cdot)$, where $x,y$ run over some ball $B$ containing $x_n$ and with $\{a\}\times B\subseteq U$.



\begin{defn}[Tame configurations]\label{config}
Fix integers $a\geq 0$, $b\geq 0$, a set
$$
T\subseteq K\times k^a\times \Gamma^b,
$$
and some $c\in K$. We say that $T$ is in $c$-config if there is $\xi\in k$ such that
$T$ equals the union over $\gamma\in \Gamma$ of sets

$$(c + A_{\xi,\gamma})\times U_{\gamma}$$
for some $U_{\gamma}\subseteq  k^a\times \Gamma^b$.
If moreover $\xi\not=0$ we speak of an open $c$-config, and if $\xi=0$ we speak of a graph $c$-config.
If $T$ is non-empty and in $c$-config, then $\xi$ and the sets $U_\gamma$ with $A_{\xi,\gamma}$ non-empty are uniquely determined by $T$ and $c$.

We say that $T \subseteq K\times k^a\times \Gamma^b$ is in $\cL$-tame config if there exist $s \geq 0$ and $\cL$-definable functions
\[g:K\to k^s\ \mbox{ and }\ c:k^s\to K\
\]
such that the range of $c$ contains no open ball,
and, for each $\eta\in k^s$, the set
\[T\cap (g^{-1}(\eta) \times k^a \times \Gamma^b  )
\]
is in $c(\eta)$-config.
\end{defn}




For any $\cL$-structure $M$ which is elementarily equivalent to $(K,\cL)$ and for any language $L$ which is obtained from $\cL$ by adding some elements of  $M$ (of any sort) as constant symbols, call $(M,L)$ a test pair for $(K,\cL)$.

\begin{defn}[Tameness]\label{Ktame}
We say that $(K,\cL)$ is weakly tame if the following conditions hold.
\begin{enumerate}
\item[(1)] Each $\cL$-definable set $T\subseteq K\times k^a\times \Gamma^b$ with $a\geq 0$, $b\geq 0$ is in $\cL$-tame config.

\item[(2)] 
For any $\cL$-definable function $F:X\subseteq K\to K$ there exist $s\geq 0$ and an $\cL$-definable function
$g:X\to k^s$ such that, for each $\eta\in k^s$, the restriction of $F$ to $g^{-1}(\eta)$ is s-continuous.
\end{enumerate}
We say that $(K,\cL)$ is tame when each test pair $(M,L)$ for $(K,\cL)$ is weakly tame.
Call an $\cL$-theory $\cT$ tame if for each model $\cM$ of $\cT$, the pair $(\cM,\cL)$ is tame.
\end{defn}

Recall \cite[Corollary 2.1.11]{CCL-PW}, which states that a tame theory, restricted in the sorts $\VF,\RF,\VG$, is $b$-minimal, in the sense of \cite{CLb}. In particular,
one can make use of  dimension theory for $b$-minimal structures.

%
%

\subsection{Skolem functions}

Recall that an $\cL$-structure $M$ has algebraic Skolem functions if for any $A\subseteq M$ every finite $A$-definable set $X\subseteq M^n$ admits an $A$-definable point.
Observe that this condition is equivalent to the fact that the model theoretic algebraic closure is equal to the definable closure.
More generally, for a multisorted language, we say that a structure $M$ has algebraic Skolem functions in the sort $S$ if for any $A\subseteq M$ and every finite $A$-definable set $X\subseteq S_M^n$ there is an $A$-definable point, with $S_M$ the universe for the sort $S$ in the structure $M$.

We say that a theory $T$ has algebraic Skolem functions (in the sort $S$), if each model has.
In any case, one can algebraically skolemize in the usual sense,  that is, given a theory $T$ in a language $\cL$, the algebraic skolemization of $T$ in the sort $S$ is the theory $T^s$ in an expansion $\cL^s$ of $\cL$ obtained by adding function symbols, such that $T^s$ has algebraic Skolem functions in the sort $S$ and such that $(\cL^s,T^s)$ is minimal with this property (where minimality is seen after identifying pairs with exactly the same models and definable sets), see also \cite{SkolemFunctions2004}.


\begin{lem}
\label{lem-skolemization}
Let $\cL$ a countable language extending $\cL_{\rm Basic}$ and $\cT$ a tame $\cL$-theory.
If $\cT$ has algebraic Skolem functions in the sort $\RF$, then it also has algebraic Skolem functions in the sort $\VF$.
In any case, there is a countable extension $\cL'$ of $\cL$ by function symbols on the sort $\RF$  
and an $\cL'$-theory $\cT'$ extending $\cT$ such that $\cT'$ has algebraic Skolem functions in the sort $\RF$ and hence also in the sort $\VF$. Moreover, every model of $\cT$ can be extended to an $\cL'$-structure that is a model of $\cT'$, and, $\cT'$ is tame.
\end{lem}

\begin{proof}
Since $\cT$ is tame, every finite definable (with parameters) set in the $\VF$ sort is in definable bijection with a definable set in the $\RF$ sort. The first statement follows: If $\cT$ has algebraic Skolem functions in the sort $\RF$, then also in the sort $\VF$. In general, let us algebraically skolemize the theory $\cT$ in the sort $\RF$. Denote by $\cL'$ and $\cT'$ the obtained language and theory. Clearly one may take $\cL'$ to be countable. It remains to prove that $\cT'$ is tame. One needs to check condition $(1)$ and $(2)$ of Definition \ref{Ktame}. Assume that $(K,\cL')$ is a model of $\cT'$ and  let $T\subseteq K\times k^a\times \Gamma^b$ be some  $\cL'$-definable set. Then there is an $\cL$-definable set $T_0$ such that $T\subseteq T_0$ and for each $(x,\xi_0,\alpha)\in T_0$, there is $\xi$ such that $(x,\xi,\alpha)\in T$ and $(x,\xi_0,\alpha)\in \mathrm{acl}_{\cL}(x,\xi,\alpha)$. Indeed, an $\cL$-formula for $T_0$ is made from one for $T$ by replacing each occurrence of a new function symbol by a formula for the definable set it lands in.  The fact that $T_0$ is in $\cL$-tame config then implies that $T$ is in $\cL'$-tame config. The reasoning for $(2)$ is similar.
\end{proof}

\begin{remark}
\label{rem-skolem-localfields}
Let $\cL$ be an extension of $\cL_{\rm Basic}$ such that any local field can be endowed with an $\cL$-structure. Let $\cT$ be an $\cL$-theory such that any ultraproduct of local fields which is of residue characteristic zero is a model of $\cT$. Consider the algebraic Skolemization $\cL'$, $\cT'$ in the sort $\RF$ from Lemma \ref{lem-skolemization}. Then one can endow  every local field with an $\cL'$-structure such that moreover any ultraproduct of such structures that is of residue characteristic zero is a model of $\cT'$.
Indeed, for each new function symbol in $\cL' \setminus \cL$ set the function output to be $0$ if the corresponding set is empty, and to be any point in the the set if non-empty.
Such a choice of $\cL'$-structure is often highly non-canonical and is not required to be compatible among field extensions.
\end{remark}


\begin{remark}
Usually the Skolemization process breaks most of the model-theoretic properties of the theory. However, since we apply it only to the residue field many results such as cell decomposition are preserved. Moreover, since we add only algebraic Skolem functions in the sort $\RF$, the situation is somehow controlled, for example, if the theory of the residue field is simple in the sense of model theory, then adding algebraic Skolem functions in the residue field preserves simplicity, see \cite{SkolemFunctions2004}.

It also worth to note that we will apply our results in the case where the residue field is pseudo-finite, and that such fields almost always have algebraic Skolem functions, see Beyarslan-Hruskovski \cite{Pseudofinite2012}. See also the work by Beyarslan-Chatzidakis \cite{geometricpsf2017} for a more concrete characterization.
\end{remark}

\subsection{Lipschitz continuity}\label{sec-lipschitz}

We can now state our first main result on Lipschitz continuity, going from local to piecewise global (with finitely many pieces).

\begin{thm}\label{Lip}
Suppose that $(K,\cL)$ is tame with algebraic Skolem functions in the sort $\RF$.
Let $f:X\subseteq K^n\to K$ be an $\cL$-definable function which is locally $1$-Lipschitz. Then there exists a finite definable partition of $X$ such that the restriction of $f$ on each of the parts is $1$-Lipschitz.
\end{thm}

As in \cite{CCL-PW}, Theorem \ref{Lip} is complemented by Theorem \ref{Lipcenter} about simultaneous partitions of domain and range into parts with $1$-Lipschitz centers. They are proved by a joint induction on $n$.

\begin{thm}[Lipschitz continuous centers in domain  and range]\label{Lipcenter}
Suppose that $(K,\cL)$ is tame with algebraic Skolem functions in the sort $\RF$.
Let $f:A\subseteq K^n\to K$ be an $\cL$-definable function which is locally $1$-Lipschitz. Then, for a finite partition of $A$ into definable parts, the following holds for each part $X$.
There exist $s\geq 0$, a coordinate projection $p:K^n\to K^{n-1}$ and
$\cL$-definable functions
$$
g:X\to k^s,\
c:  p(X)\subseteq K^{n-1}\to K \mbox{ and } d: p(X)\subseteq K^{n-1}\to K
$$
such that, $c$ and $d$ are 1-Lipschitz, and for each $\eta\in k^s$ and $w$ in $p(K^{n})$, the set $g^{-1}(\eta)_w$ is in $c(w)$-config and the image of $g^{-1}(\eta)_w$ under $f_w$ is in $d(w)$-config.
\end{thm}

Before proving Theorems \ref{Lip} and \ref{Lipcenter}, we establish in Lemma \ref{lem-centers-indep} a weaker version of Theorem \ref{Lipcenter}, where the centers are only required to be locally 1-Lipschitz. It will itself rely on \cite[Theorem 2.1.8]{CCL-PW}, which looks similar but there the centers depend on auxiliary parameters.

\begin{lem}\label{lem:loc1Lipschitz}
Suppose that $(K,\cL)$ is tame with algebraic Skolem functions in the sort $\RF$. Let $Y\subseteq K^n\times k^s$ be a definable set, $p : Y\to K^n$ be the canonical projection, $X=p(Y)$, and $f : X \to K$ be a definable function such that for each $\eta\in k^s$, the restriction of $f$ to $Y_\eta$ is  locally 1-Lipschitz. Then there is a finite definable partition of $X$ such that the restriction of $f$ on each of the pieces is locally 1-Lipschitz.
\end{lem}
The proof of Lemma \ref{lem:loc1Lipschitz} is  a joint induction with the following lemma.
\begin{lem}\label{lem:projection-full-dim}
Suppose that $(K,\cL)$ is tame with algebraic Skolem functions in the sort $\RF$.
Let $A\subseteq K^m$ be a definable set of dimension $n$. Then there is a finite definable partition of $A$ such that for each part $X$, there is an injective projection $X\subseteq K^m \to K^n$ and its inverse is locally 1-Lipschitz.
\end{lem}
\begin{proof}[Proof of Lemma \ref{lem:projection-full-dim}]
Assume Lemma \ref{lem:loc1Lipschitz} holds for integers up to $n$.
We will use dimension theory for $b$-minimal structures. We get a finite definable partition of $A$ such that on each piece $X$, there is a projection $p : X\to K^n$ which is finite-to-one. For each $w\in p(X)$, the fiber $X_w$ is finite. By the existence of algebraic Skolem functions in the sort $\RF$ and hence also in $\VF$ by Lemma \ref{lem-skolemization}, each of the points of $X_w$ is definable. By compactness, we can find a finite definable partition of $X$ such that $p$ is injective on each of the pieces.

By \cite[Corollary 2.1.14]{CCL-PW}, up to changing the coordinate projection we see that the inverse of $p$ is locally 1-Lipschitz when restricted to fibers of some definable function $g : p(X)\to k^{r}$. By Lemma \ref{lem:loc1Lipschitz}, we can find a finite partition of $p(X)$ such that the inverse of $p$ is locally 1-Lipschitz on each of the parts.
\end{proof}
\begin{proof}[Proof of Lemma \ref{lem:loc1Lipschitz}]
We work by induction on $n$. If $n=0$ there is nothing to prove. Assume now $n\geq 1$ and that Lemmas \ref{lem:loc1Lipschitz} and \ref{lem:projection-full-dim} hold for integers up to $n-1$. Assume first that $X$ is of dimension $n$. By dimension theory, there is at least one $\eta$ such that $Y_\eta$ is of dimension $n$. Define $X'$ to be the union of the interior of $Y_\eta$ for all such $\eta\in k^s$. The function $f$ is locally 1-Lipschitz on $X'$. It remains to deal with $X''=X\backslash X'$. By dimension theory, $X''$ is of dimension less than $n$. Assume $X''=X$ for simplicity. By Lemma \ref{lem:projection-full-dim}, up to considering a finite definable partition of $X$ we can assume that there is an injective coordinate projection $p : X \to K^{n-1}$ with inverse locally 1-Lipschitz. Then $f$ is locally 1-Lipschitz if and only if $f\circ p^{-1}$ is. Now $p(X)$ with the function $f\circ p^{-1}$ satisfies the hypothesis of Lemma \ref{lem:loc1Lipschitz}. By induction hypothesis, we have the result.
\end{proof}

\begin{lem}\label{lem-centers-indep}
Suppose that $(K,\cL)$ is tame with algebraic Skolem functions in the sort $\RF$.
Let $f:A\subseteq K^n\to K$ be an $\cL$-definable function which is locally $1$-Lipschitz. Then, for a finite partition of $A$ into definable parts, the following holds for each part $X$.
There exist $s\geq 0$, a coordinate projection $p:K^n\to K^{n-1}$ and
$\cL$-definable functions
\[
g:X\to k^s,\
c: p(X)\subseteq K^{n-1}\to K \mbox{ and } d: K^{n-1}\to K
\]
such that the functions $c$ and $d$ are locally $1$-Lipschitz, and, for each $w$ in $p(K^{n})$, the set $g^{-1}(\eta)_w$ is in $c(w)$-config and the image of $g^{-1}(\eta)_w$ under $f_w$ is in $d(w)$-config.
\end{lem}
The proof uses \cite[Theorem 2.1.8]{CCL-PW}, but only a weaker version is actually needed: we only need to require the centers to be locally 1-Lipschitz.
\begin{proof}
Apply Theorem  \cite[Theorem 2.1.8]{CCL-PW} to $f$. Work on one of the definable pieces $X$ of $A$ and use notations from the application of Theorem  \cite[Theorem 2.1.8]{CCL-PW}, which is similar to \ref{Lipcenter} except that the input of $c$ and $d$ may additionally depend on some $k$-variables.  We now show that these additional $k$-variables are not needed as input for $c$ and $d$. We first show (after possibly taking a finite definable partition of $X$) that $c(\cdot,w)$ and $d(\cdot,w)$ are constant.

Fix some $w\in p(X)$. Since the range of the $w$-definable function $c_w : \eta \in k^s \mapsto c(\eta,w)\in K$ does not contain an open ball, it must be finite. By tameness, there is a $w$-definable bijection $h_w$ between the range of $c_w$ and a subset of $B_w\subseteq k^{s'}$, for some $s'\in \NN$. By the existence of algebraic Skolem functions in the sort $\RF$ and hence also in $\VF$ by Lemma \ref{lem-skolemization}, each of the points of $B_w$ is $w$-definable. Taking the preimage of those points by $h_w\circ c_w$ leads to a $w$-definable finite partition of $k^s$. After taking preimages by $g$, it itself leads to a finite $w$-definable partition of $X_w$.  By compactness, we find a finite partition of $X$ such that on each piece, the function $c(g(x),p(x))$ is independent of $g(x)\in k^s$ and can be (abusively) written $c(p(x))$. The argument for $d$ is similar.

By Lemma \ref{lem:loc1Lipschitz}, we can refine the partition such that the functions $c,d : p(X)\to K$ are locally 1-Lipschitz.
\end{proof}


\begin{proof}[Proof of Theorem \ref{Lipcenter}] We proceed by induction on $n$.
Theorem \ref{Lipcenter} for $n=1$ is exactly Lemma \ref{lem-centers-indep} for $n=1$ since the Lipschitz condition is empty in this case. Assume now that Theorems \ref{Lip} and \ref{Lipcenter} hold for integers up to $n-1$. Apply Lemma \ref{lem-centers-indep}. On each of the definable pieces $X$ obtained, one has a coordinate projection $p$ and definable functions $c,d : p(X)\to K$ that are locally 1-Lipschitz. By Theorem \ref{Lip} for $n-1$, we have a finite definable partition of $p(X)$ such that $c$ and $d$ are 1-Lipschitz on each of the pieces. This induces a finite definable partition of $X$ satisfying the required properties.
\end{proof}

\begin{proof}[Proof of Theorem \ref{Lip}]
We work by induction on $n$, assuming  that Theorem \ref{Lipcenter} holds for integers up to $n$ and Theorem \ref{Lip} holds for integers up to $n-1$. For $n=0$ there is nothing to show, hence we assume $n\geq 1$. Write $p : X \to K^{n-1}$ for the coordinate projection sending $x=(x_1,\dots,x_n)$ to $\hat x=(x_1,...,x_{n-1})$, and define $Y$ as the image of $X$ under the function $ h : X\to K^n$ sending $x$ to $(\hat x,f(x))$.

Up to taking a finite definable partition of $X$, switching the variables, by induction on the number of variables on which $f$ depends, by Lemma \ref{lem:projection-full-dim} and Theorem \ref{Lipcenter}, tameness and compactness, we may assume that the following holds :
\begin{itemize}
\item X is open in $K^n$,
\item there is a definable function $g : X \to k^s$, and definable functions $c,d : p(X) \to K$,
\item for each $\hat x\in p(X)$ and $\eta\in k^s$,  $g^{-1}(\eta)_{\hat x}$ is in open $c(\hat x)$-config, $h(g^{-1}(\eta))_{\hat x}$ is in $d(\hat x)$-config,
\item the restriction of $f(\hat x,\cdot)$ to $g^{-1}(\eta)_{\hat x}$ is s-continuous for each $\hat x\in p(X)$ and $\eta\in k^s$,
\item the functions $c$ and $d$ are 1-Lipschitz,
\item the function $f(\cdot, x_n)$ is 1-Lipschitz for each $x_n$.
\end{itemize}
We show that under these assumptions, $f$ is 1-Lipschitz. Since $d$ is 1-Lipschitz, we can replace $f$ by $x\mapsto f(\hat x,x_n)-d(\hat x)$ (and translate $Y$ accordingly) in order to assume $d=0$.

Let $x,y\in X$ and assume first that both $x_n$ and $y_n$ lie in an open ball $B\subseteq X_{\hat x}$. Then $g(x)=g(\hat x, y_n)$, indeed otherwise $c(\hat x)\in B$, which would contradict that $g^{-1}(\eta)_{\hat x}$ is in open $c(\hat x)$-config for every $\eta\in k^s$. It follows that $f(\hat x, \cdot)$ is $s$-continuous on $B$. Since $f$ is locally 1-Lipschitz, the constant $\gamma$ involved in the definition of s-continuity on $B$ satisfies $\gamma\leq 1$.

Thus, using the ultrametric inequality and the assumption about $f(\cdot,y_n)$, we have :
\begin{eqnarray*}
 |f(x) - f(y)|
 & = & |f(x) - f(\hat x,y_n) + f(\hat x,y_n) - f(y)| \\
 & \leq & \max(|f(x) - f(\hat x,y_n)|,\  |f(\hat x,y_n) - f(y)| ) \\
 & \leq & \max(|x_n - y_n|,\  |\hat x - \hat y| ) \\
 & = & | x- y |,
   \end{eqnarray*}
which settles this case.

Suppose now that $x_n$ and $y_n$ do not lie in an open ball included in  $X_{\hat x}$, and by symmetry neither in an open ball included in $X_{\hat y}$. This implies that
\begin{equation}\label{equi000}
 |x_n-c(\hat x)| \leq |x_n - y_n| \mbox{ and }  |y_n-c(\hat y)|   \leq |x_n - y_n| .
\end{equation}

By s-continuity and the fact that $f$ is locally 1-Lipschitz, the image of a small enough open ball in $X_{\hat x}$ of radius $\alpha$ is either a point or an open ball of radius less or equal to $\alpha$. This implies that
\begin{equation}\label{equi00}
|f(x)-d(\hat x) | \leq   |x_n-c(\hat x)|  \mbox{ and }   |f(y) -d(\hat y)| \leq   |y_n - c (\hat y) |.
\end{equation}

Recall that $d=0$. Combining (\ref{equi000}) and (\ref{equi00}), we have by the ultrametric inequality
\[
|f(x) - f(y)|  \leq \max (|x_n-c(\hat x)|, |y_n - c (\hat y) | ) \leq | x_n- y_n | \leq | x- y |,
\]
which end the proof.
\end{proof}

\begin{remark}
Let us recall that \cite{cluckers_lipschitz_2010} and \cite{cluckers_approximations_2012}, with related results on Lipschitz continuity on $p$-adic fields, are amended in Remark 2.1.16 of \cite{CCL-PW}. When making $d=0$ it is important to keep $c$ possibly nonzero in the proof of \cite[Theorem 2.1.7]{CCL-PW} and in the above proof of Theorem \ref{Lip}; this was forgotten in the proofs of the corresponding results \cite[Theorems 2.3]{cluckers_lipschitz_2010} and \cite[Theorem 3.5]{cluckers_approximations_2012}, where $c$ should also have been kept.
\end{remark}

\section{Analytic parametrizations}
\label{section_analytic_parametrizations}

The goal of this section is to prove a uniform version of non-Archimedean Yomdin-Gromov parametrizations.

\subsection{$T_r$-approximation}
\begin{setting}\label{setting-T_DP}
We fix for the whole section one of the two following settings, of $\cT_\mathrm{DP}$, or, $\cT^\mathrm{an}_\mathrm{DP}$, both of which we now introduce. Let $\cO$ be the ring of integers of a number field. Recall that the Denef-Pas language is a three sorted language, with one sort $\VF$ for the valued field with the ring language, one sort $\RF$ for the residue field with the ring language, one sort $\VG$ for the value group with the Presburger language with an extra symbol for $\infty$, and function symbols $\ord : \VF\mapsto \VG$ for the valuation (sometimes denoted multiplicatively $\abs{\cdot}$) and $\ac : \VF\to \RF$ for an angular component map (namely a multiplicative map sending $0$ to $0$ and sending a unit of the valuation ring to its reduction modulo the maximal ideal). Consider the theory of henselian discretely valued fields of residue field characteristic zero in the Denef-Pas language, with constants symbols from $\cO\llb t\rrb$ and with $t$ as a uniformizer of the valuation ring. This theory is tame by Theorem 6.3.7 of \cite{cluckers_fields_2011}. Applying Lemma \ref{lem-skolemization}, one obtains a new language and  a new theory which we denote by
$\cL_\mathrm{DP}$ and $\cT_\mathrm{DP}$, which thus has algebraic Skolem functions in each of the sorts.

We can also work in an analytic setting, as follows. Consider the expansion of the Denef-Pas language $\cL_\mathrm{DP}$ by adding function symbols for elements of
\[\cO\llb t\rrb\{x_1,\dots,x_n\}=\set{f=\sum_{I\in \NN^n} a_I x^I\mid a_I\in \cO\llb t\rrb,  \ord_t(a_I)\underset{\abs{I}\to +\infty}{\to} +\infty}.
\]
Any complete discretely valued field over $\cO$ (namely, with a unital ring homomorphism from $\cO$ into the valued field) can be endowed with a structure for this expansion, by interpreting the new function symbols as the corresponding power series evaluated on the unit box and put equal to zero outside the unit box.
Let $\cL_\mathrm{DP}^\mathrm{an}$ and $\cT^\mathrm{an}_\mathrm{DP}$ the resulting language, resp.~the theory of these models. (For a shorter and explicit axiomatization for the analytic case,  see the axioms of Definition 4.3.6(i) of \cite{cluckers_fields_2011}.)

For now on, we work in a language $\cL$ that is either $\cL_\mathrm{DP}$ or $\cL_\mathrm{DP}^\mathrm{an}$ and in the theory $\cT$ that is correspondingly $\cT_\mathrm{DP}$ or $\cT^\mathrm{an}_\mathrm{DP}$.

Let us summarize our theory once more: $\cT$ is the $\cL$-theory which is the algebraic skolemization in the residue field sort of the theory of complete discrete valued fields, residue field of characteristic zero, with constants symbols from $\cO\llb t\rrb$ (as a subring) and where $t$ has valuation $1$, and (in the subanalytic case), with the restricted analytic function symbols as the corresponding power series evaluated on the unit box and put equal to zero outside the unit box.

In any case, the theory $\cT$ is tame by Theorem 6.3.7 of \cite{cluckers_fields_2011}, and, it has algebraic Skolem functions in each sort by Lemma \ref{lem-skolemization} and by Example 4.4(1) with the homothecy with factor $t$ on the valuation ring to make the system strict instead of separated. Note that there is no need to algebraically skolemize again when going from $\cT_\mathrm{DP}$ to the larger theory $\cT^\mathrm{an}_\mathrm{DP}$ by the elimination of valued field quantifiers from Theorem 6.3.7 of \cite{cluckers_fields_2011}.
Definable means definable without parameters in the theory $\cT$.
\end{setting}

\begin{defn}[$T_r$-approximation] Let $L$ be any valued field. Consider a set $P\subseteq L^m$, a function $f=(f_1,\dots,f_n) : P\to \cO_L^n$ and an integer $r > 0$.
We say that $f$ satisfies \emph{$T_r$-approximation} if $P$ is open in $L^m$, and, for each $y\in P$, there is an $n$-tuple $T_{f,y}^{<r}$ of polynomials with coefficients in $\cO_L$ and of degree less than $r$ that satisfies, for all $x\in P$,
\[
\abs{f(x)-T_{f,y}^{<r}(x)}\leq \abs{x-y}^r.
\]

We say that a family $(g_i)_{i\in I}$ of functions $g_i : P_i \to X_i  \subseteq  \cO_L^n $ is a $T_r$-parametrization of $X=\bigcup_{i\in I} X_i$ if each $g_i$ is surjective and satisfies $T_r$-approximation.
\end{defn}
Observe that if $f$ satisfies $T_r$-approximation, then the polynomials $T_{f,y}^{<r}$ are uniquely determined.

Observe also that if $K$ is a complete valued field of characteristic zero, if $f$ is of class $\cC^r$ and satisfies  $T_r$-approximation, then $T_{f,y}^{<r}$ is just the tuple of Taylor polynomials of $f$ at $y$ of order $r$.

\begin{notation}\label{sec:not}
Let $\cO$ be the ring of integers of a number field. We denote by $\cA_\cO$ the collection of all local fields of characteristic zero over $\cO$, $\cB_\cO$ the collection of all local field of positive characteristic  over $\cO$, and set $\cC_\cO=\cA_\cO\cup \cB_\cO$. (By a local field $L$ over $\cO$ we mean a non-archimedean locally compact field, hence, a finite field extension of $\QQ_p$ or of $\FF_p\llp t\rrp$ for a prime $p$, allowing a unital homomorphism $\cO\to L$.) If $L\in \cC$, we denote by $\ord$ its valuation (normalized such that $\ord(L^\times)=\ZZ$), $\cO_L$ its valuation ring, $\cM_L$ its maximal ideal, $\varpi_L\in \cM_L$ a fixed choice of uniformizer, $k_L$ its residue field, $q_L$ the cardinal of $k_L$ and $p_L$ the characteristic of $k_L$. If $N\in \NN$, we define $\cA_{\cO,N}$ (resp. $\cB_{\cO,N}$, resp. $\cC_{\cO,N}$) to be the set of $L\in \cA_{\cO}$ (resp. $L\in\cB_{\cO,N}$, resp. $L\in \cC_{\cO,N}$) such that $p_L\geq N$. By Remark \ref{rem-skolem-localfields}, we can consider $L\in \cC_\cO$ as an $\cL$-structure, and any non-principal ultraproduct of such local fields is a model of $\cT$.
\end{notation}

Call a family of definable sets a definable family, if the index set and the total set are both definable, namely, a family $(X_y)_{y\in Y}$ of definable sets $X_y$ indexed by $y\in Y$ is called a definable family if $Y$ and the total set $\cX:=\{(x,y)\mid x\in X_y,\ y\in Y\}$ are definable sets. Likewise, a family of definable functions is called a definable family if the family of graphs is a definable family of definable sets.
We use notations like $\cO_\VF$ for the definable set which in any model $K$ is the valuation ring $\cO_K$, and similarly $\cM_\VF$ for the maximal ideal, and so on. For a definable set $X$ and a structure $L$, we will write $X(L)$ for the $L$-points on $X$, and, for a definable function $f:X\to Y$ we will write $f_L$ for the corresponding function $X(L)\to Y(L)$.\footnote{When we interpret definable sets or functions into local fields $L$ (or, more generally, $\cL$-structures that are not models of our theory $\cT$), we implicitly assume that we have chosen some formula $\varphi$ that defines the set and consider $\varphi(L)$. This set $\varphi(L)$ may of course change with a different choice of formula $\varphi$ for small values of the residue field characteristic of $L$, but this is not a problem by Remark \ref{rem-skolem-localfields}, and since we are interested only in the case of large residue field characteristic.}

The main goal of this section is to prove the following two theorems on the existence of $T_r$-parameterizations with rather few maps, in terms of $r$. Even the mere finiteness of the parameterizing maps is new, as compared to \cite{CCL-PW} where `residue many' maps were allowed, but we even get an upper bound which is polynomial in $r$. Recall from Setting \ref{setting-T_DP} that we work in a theory with algebraic Skolem functions.

\begin{thm}[Uniform $T_r$-approximation in local fields]\label{uniform-Tr-approx}
Let $n\geq 0$, $m\geq 0$ be integers and let $X=(X_y)_{y\in Y}$ be a definable family of definable subsets $X_y \subseteq \cO_\VF^n$, for $y$ running over a definable set $Y$. Suppose that $X_y$ has dimension $m$ for each $y\in Y$ (and in each model of $\cT$). Then there exist integers $c>0$ and $M>0$ such that for each $L\in \cC_{\cO,M}$ and for each integer $r>0$
, there are a finite set $I_{r,q}$ of cardinality $cr^m$
and a definable family $g=(g_{y,i})_{(y,i)\in Y(L)\times I_r}$ of definable functions
$$
g_{y,i}:P_{y,i}   \to X_y(L)
$$
with $P_{y,i} \subseteq \cO_L^{m}$ such that for each $y\in Y(L)$, the family $(g_{y,i})_{i\in I_{r,q}}$ forms a $T_r$-parametrization of $X_y(L)$.
\end{thm}


The following result is uniform in all models $K$ of $\cT$. 
Note that $\cT$ requires in particular the residue field to have characteristic zero, and the value group to be elementarily equivalent to $\ZZ$.

\begin{thm}[Uniform $T_r$-approximation for models of $\cT$]\label{uniform-Tr-approx-char0}
Let $n\geq 0$, $m\geq 0$ be integers and let $X=(X_y)_{y\in Y}$ be a definable family of definable subsets $X_y \subseteq \cO_\VF^n$, for $y$ running over a definable set $Y$. Suppose that $X_y$ has dimension $m$ for each $y\in Y$ and each model of $\cT$. Then there exists an integer $c>0$ such that for each model $K$ of $\cT$ and for each integer $r>0$ such that the $r$-th powers in the residue field have a finite number $b_r=b_r(K)$ of cosets, there are a finite set $I_{r}$ of cardinality $c (b_r r)^m$ and a $R_r$-definable family $g=(g_{y,i})_{(y,i)\in Y(K)\times I_r}$ of $R_r$-definable functions
$$
g_{y,i}:P_{y,i}   \to X_y(K)
$$
with $P_{y,i} \subseteq \cO_K^{m}$ such that for each $y\in Y(K)$, the family $(g_{y,i})_{i\in I_{r}}$ forms a $T_r$-parametrization of $X_y(K)$ and where $R_r\subset \cO_K^\times$ is a set of lifts of representatives for the $r$-th powers in $k^\times$.
\end{thm}

\begin{remark}
Observe that even if Theorems \ref{uniform-Tr-approx} and \ref{uniform-Tr-approx-char0} are very similar, one cannot deduce the first from the second by compactness. The reason is the quantification over $r$ in the statement.  They will however both be deduced from the upcoming Theorem \ref{strong-T1-approx}, which is a $T_1$-parametrization theorem with an extra technical condition. It will allow us to define a $T_r$-parametrization by precomposing by power functions. Furthermore, note that in Theorem \ref{uniform-Tr-approx}, the factor  $b_r$ for the index of $r$th powers in the residue field is not needed; this is because of an additional trick using a property true in  finite fields. \end{remark}

\begin{remark}
For most of the section, we could in fact work in a slightly more general setting (up to imposing some additional requirements for Theorem \ref{uniform-Tr-approx}). Using resplendent relative quantifier elimination as in \cite{Rid}, we can add arbitrary constants symbols and allow an arbitrary residual extension (and an arbitrary extension on the value group) of the language and the theory before applying the algebraic Skolemization in the residue field sort. In particular, \ref{uniform-Tr-approx-char0} holds in this more general setting. If the extended language and theory still have the property that any local field can be equipped with a structure for the extended language such that moreover any ultraproduct of such equipped local fields which  is of residue characteristic zero is a model of the extended theory, then also Theorem \ref{uniform-Tr-approx} would go through.
\end{remark}

\begin{remark}
The condition that the value group be a Presburger group can probably be relaxed to any value group in which the index $v_r$ of the subgroup of $r$-multiples is finite, by replacing $c (b_r r)^m$ by $c(b_rv_r)^m$ for the cardinality of $I_{r}$ and taking  $R_r\cup V_r$ instead of $R_r$ with $V_r$ a set of lifts of representatives for the $r$-multiples in the value group.

Note that an adaptation of Theorem \ref{uniform-Tr-approx-char0}  and its proof to mixed characteristic henselian valued fields may be possible too, with the adequate adaptations. For example, when going from local to piecewise Lipschitz continuous, the Lipschitz constant should be allowed to grow. (Indeed, look at the function $x\mapsto x^p$ on the valuation ring of $\CC_p$.)
\end{remark}

Before starting the proofs of Theorems \ref{uniform-Tr-approx} and \ref{uniform-Tr-approx-char0}, we need a few more definitions.

\begin{defn}[Cell with center]\label{cellc}
Consider an integer $n\geq 0$.
For non-empty definable sets $Y$ and $X\subseteq Y\times \VF^n$, the set $X$ is called a cell over $Y$ with center $(c_i)_{i=1, \cdots, n}$ if it is of the form
$$
\{(y,x) \in Y\times \VF^n\mid y\in Y,\ \ac(x_i-c_i(x_{<i}))=\xi_{i}(y),\, (y,(|x_i-c_i(x_{<i})|)_i)\in G  \},
$$
for some set $G\subseteq Y\times \VG^n$ and some definable functions $\xi_i: Y\to k$ and $c_i:Y\times \VF^{i-1}\to \VF$, where $x_{<i}=(y,x_1,\ldots,x_{i-1})$. If moreover $G$ is a subset of $Y\times (\VG^\times)^n$, where $(\VG^\times)^0=\{0\}$, then $X$ is called an open cell over $Y$
(with center $(c_i)_{i=1, \cdots, n}$).
\end{defn}

\begin{defn}[Cell around zero]
We say that $X\subseteq \VF^n$ is a cell around zero if it is of the form
\[
X=\set{x=(x_1,\dots,x_n)\in \VF^n\mid \ac(x)\in A, (\abs{x_1},\cdots,\abs{x_n})\in B}
\]
for some definable sets $A\subseteq \RF^n$ and $B\subseteq \VG^n$.
Similarly one can call a set $X$ a cell around zero for $X\subset L^n$ for some valued field $L$  with an angular component map, if it is of the corresponding form.
\end{defn}

\begin{defn}[Associated cell around zero]\label{cell0}
Let $X$ be a cell over $Y$ with center, with notation from Definition \ref{cellc}. The cell around zero associated to $X$ is by definition the cell $X^{(0)}$ obtained by forgetting the centers, namely
$$
X^{(0)}=\{(y,x) \in Y\times \VF^n\mid y\in Y,\ \ac(x_i)=\xi_{i}(y),\, (y,(|x_i|)_i)\in G  \}
$$
with associated bijection  $\theta_X:X\to X^{(0)}$ sending $(y,x)$ to $(y,(x_i-c_i(x_{<i}))_i)$.
For a definable map $f:X\to Z$ there is the natural corresponding function $f^{(0)}= f\circ \theta_X^{-1}$ from $X^{(0)}$ to $Z$.
\end{defn}

%

We now define the term language. This is an expansion $\cL^*$ of $\cL$, by joining division and witnesses for henselian zeros and roots.

\begin{defn} \label{hf1}
Let $\cL^*$ be the expansion of $\cL\cup\{^{-1}\}$ obtained by joining to $\cL\cup\{^{-1}\}$ function symbols $h_{m}$ and $\mathrm{root}_m$ for integers $m>1$, where on a henselian valued field $K$ of equicharacteristic zero and residue field $k$ these functions are the functions
\[
h_{m}:K^{m+1}\times k \to K
\]
sending $(a_0,\ldots,a_{m},\xi)$ to the unique $y$ satisfying
$\ord(y)=0$, $\ac(y)\equiv \xi\bmod \cM_K$, and
$\sum_{i=0}^{m} a_{i} y^i=0$, whenever $\xi$ is a unit,
$\ord(a_i)\geq 0$, $\sum_{i=0}^m a_{i} \xi^i\equiv 0\bmod \cM_K$, and
 \[
 f'(\xi)\not\equiv0\bmod \cM_K
 \]
with $f'$
the derivative of $f$, and to $0$ otherwise. Likewise, $\mathrm{root}_m$ is the function $K\times k\to K$ sending $(x,\xi)$ to the unique $y$ with $y^m=x$ and $\ac(y)=\xi$ if there is such $y$, and to zero otherwise.
\end{defn}

\begin{prop}[Term structure of definable functions]
\label{terms} Every $\VF$-valued definable function is piecewise given by a term.
More precisely, given a definable set $X$ and a definable function $f:X\to \VF$, there exists a finite partition of $X$ into definable parts and for each part $A$ an $\cL^*$-term $t$ such that
$$
t(x) = f(x)
$$
for all $x\in A$.
\end{prop}
\begin{proof}
By Theorem 7.5 of \cite{CLR_analytic_2006} there exists a definable function $g:X\to \RF^{m}$ for some $m\geq 0$ and an $\cL^*$-term $t_0$ such that
$$
t_0(x,g(x)) = f(x).
$$
Since the terms $h_n$ (the henselian witnesses) and $\mathrm{root}_n$ (the root functions) involve at most a finite choice in the residue field, one  can reduce to the case that $g$ has finite image. The fibers of $g$ can then be taken as part of the partition to end the proof.
\end{proof}

\subsection{Condition $(*)$}\label{sec:cond*}
We now introduce a technical condition, named $(*)$, that will be used in Section \ref{sec:analyci} to show a strong form of analyticity of definable functions, named global analyticity in Definition \ref{def:globalan}.

\begin{defn}[Condition $(*)$]
\label{def-condition-star}
We first define condition $(*)$ for $\cL^*$-terms, 
inductively on the complexity of terms.
Consider a definable set $X\subseteq \VF^m$ and let $x$ run over $X$.

We say that a $\VF$-valued $\cL^*$-term $t(x)$ satisfies condition $(*)$ on $X$ if the following holds.

If $t(x)$ is a term of complexity 0 (\emph{i.e.} a constant or a variable), then it satisfies condition $(*)$ on $X$.

Suppose now that the term $t$ is either $t_1+t_2$, $t_1\cdot t_2$, $t_0^{-1}$, $h_n(t_0,\dots,t_n;t_{-1})$, $\mathrm{root}_n(t_0;t_{-1})$ for some $n>0$, or of the from $\underline f(t_1,\dots,t_n)$, with $\underline f$ one of the analytic functions of the language. In the first two cases, we just require that $t_1$ and $t_2$ satisfy condition $(*)$ on $X$. In the remaining four cases, we require that $t_0,\dots, t_n$ satisfy condition $(*)$ on $X$ and moreover that for any box $B\subseteq X$, the functions $t_{-1}$ and $\ac(t_0),\dots,\ac(t_n), \ord(t_0),\dots,\ord(t_n)$ are constant on $B$.

We finally say that an $\cL$-definable function $f : X\subseteq \VF^m\to \VF^{m'}$ for $m' > 0$ satisfies condition $(*)$ on $X$ if there is a tuple $t$ of  $\cL^*$-terms $t_i(x)$ satisfying condition $(*)$ on $X$ and such that $f(x)=t(x)$ for $x\in X$.
\end{defn}

The following lemma ensures existence of functions satisfying condition $(*)$.
\begin{lem}
\label{lem-conditionstar}
Let $f : X\subseteq Y\times \VF^m \to \VF^{m'}$ be a definable function for some $m$ and $m'$. Then there is a finite partition of $X$ into some open cells $A$ over $Y$ with center $(c_i)_{i=1, \cdots, m}$  and a set $B$ such that $B_y$ is of dimension less than $m$ for each $y\in Y$,  such that the function
$$
(A^{(0)})_y\to \VF^{m'} : x\mapsto f^{(0)}(y,x)
$$
satisfies condition $(*)$ on  $(A^{(0)})_y$ for each $y$, with notation from Definition \ref{cell0}.
\end{lem}
\begin{proof}
We proceed by induction on $m$.
By Proposition \ref{terms} for $f$ we may suppose that $f$ is given by a tuple $t(x)$ of $\cL^*$-terms. Let $h:X\to \RF^{s}\times \Gamma^{s'}$ be the following definable function created from $t$: $h$ has a  component function of the form $t'$ for each $\RF$-valued subterm $t'$ of $t$ and also of the forms $\ord(t'')$ and $\ac(t'')$ for each $\VF$-valued subterm $t''$ of $t$.
The proposition requires us to find a finite partition of $X$ into cells over $Y$ such that for each open cell $A$ over $Y$, the map $(f_{|A})^{(0)}(y,\cdot)$ has condition $(*)$ on  $A^{(0)}_y$, with notation from Definition \ref{cell0}. 
Now apply the cell decomposition theorem adapted to $h$ and  work on one of the open pieces $A$. Thus, $A$ is an open cell over $Y$ with some center $(c_i)_{i=1, \cdots, m}$ adapted to $h$, namely, there are definable functions
$c_i : A^i\subseteq \VF^i\to \VF$ for $i=0,\dots, m-1$ such that
$h^{(0)}$ is constant on each box contained in $c^{-1}(A)$, which is moreover an open  cell around zero, where
\[
c : x\in \VF^m\mapsto (x_1+c_0,x_2+c_1(x),\dots, x_m+c_{m-1}(x)),
\]
with notation from Definition \ref{cell0}. Note that $c=\theta_A^{-1}$ and $c^{-1}(A)= A^{(0)}$ in that notation.
\end{proof}

\begin{defn}[Associated box]\label{defn:assball}
Let $K$ be a valued field.
By a box $B\subset K^n$ we mean a product of open balls in $K$. Let $B=\prod_{1\leq i \leq n}B(a_i,r_i)\subseteq K^n$ be a box, with open balls
$$
B(a_i,r_i) = \set{x\in K\mid \abs{x-a_i}<r_i}
$$
with $a_i\in K$ and nonzero $r_i\in \Gamma_K$. The box associated to $B$ is the box $B_{\rm as}\subseteq K^{\rm alg}$ defined by
\[
B_{\rm as}=\set{x\in (K^{\rm alg})^n\mid \abs{x-a_i}<r_i},
\]
where $K^{\rm alg}$ is an algebraic closure of $K$, endowed with the canonical extension of the valuation of $K$.
\end{defn}

\subsection{Global analyticity}\label{sec:analyci}
To easier speak of analyticity in this section, we will work with complete discretely valued fields (a meaning of analyticity exists for all models of $\cT$ by \cite{cluckers_fields_2011}).

\begin{defn}[Globally analytic map]\label{def:globalan}
Let $K$ be a complete discretely valued field.
Let $X\subseteq K^m$ be a set and $f : X\to K^n$ a function. We say that $f$ is globally analytic on $X$ is for each box $B\subseteq X$, the restriction of $f$ to $B$ is given by a tuple of power series with coefficients in $K$, (say, taken around some $a\in B$), which converges on the associated box $B_{\rm as}$.\footnote{Here, converging on $B_{\rm as}$ means that the partial sums obtained by evaluating at any element of $B_{\rm as}$ form a Cauchy sequence (the limits actually lie inside $K^{\rm alg}$ by \cite{cluckers_fields_2011}).}
\end{defn}

The following proposition is the reason why we introduced condition $(*)$. Observe that it applies also to local fields, and thus not only to models of our theory $\cT$.

\begin{prop}[Analyticity, {\cite[Lemma 6.3.15]{cluckers_fields_2011}}]
\label{prop-condition-star-analyticity}
Let $f$ be a definable function satisfying condition $(*)$ on some definable set $X$. Then there is some $M>0$ such that for $L$ either a model of $\cT$ which is a complete discretely valued field, or, a local field with residue field cardinality at least $M$, the following holds.
For any box $B \subseteq X(L)$ and $b\in B$, there is a power series $g$ centered at $b$ and converging on $B_\as$ such that $f$ is equal to $g$ on $B$.
Moreover, $M$ can be taken uniformly in definable families of definable functions.
\end{prop}
\begin{proof}
We recall the strategy of the proof of  \cite[Lemma 6.3.15]{cluckers_fields_2011}. One works by induction on the complexity of the $\cL^*$-term corresponding to the definition of condition $(*)$, using compositions of power series as in Remark 4.5.2 of \cite{cluckers_fields_2011}. The only nontrivial cases are $t_0^{-1}$, $h_n(t_0,\dots, t_n;t_{-1})$,  $\mathrm{root}_n(t_0;t_{-1})$, and ${\underline f}(t_1,\ldots,t_n)$ for some restricted analytic function ${\underline f}$ from the language. If $L$ is a model of $\cT$, we may assume by the definition of condition $(*)$, that the terms  $t_i$ satisfy condition $(*)$ on $X$ and that $t_{-1}$ and $\ac(t_0),\dots,\ac(t_n),\ord(t_0),\dots,\ord(t_n)$ are constant on $B$. In the local field case, by compactness there is some $M>0$ such that if the residue field of $L$ is of cardinality at least $M$, the functions $t_{-1}$ and $\ac(t_0),\dots,\ac(t_n),\ord(t_0),\dots,\ord(t_n)$ are constant on any box $B$ contained in $X(L)$. One finishes exactly as in the proof of \cite[Lemma 6.3.11]{cluckers_fields_2011}, where  for the case ${\underline f}(t_0,\dots, t_n)$, with ${\underline f}$ one of the analytic functions of the language, condition $(*)$ ensures that either the function $f$ is interpreted as the zero function on a box $B$, or, the image of the box $B$ by $(t_1,\dots, t_n)$ is strictly contained in the unit box, hence so is the image of $B_\as$, ensuring convergence of $f$ on it, hence analyticity of ${\underline f}(t_0,\dots, t_n)$ on $B_\as$.
\end{proof}

\subsection{Strong $T_r$ approximation}\label{sec:strongap}

We can now state a stronger notion of $T_r$-approximation, for definable functions. The strong $T_1$-approximation will be key for the proofs of Theorems \ref{uniform-Tr-approx} and \ref{uniform-Tr-approx-char0}. Strong $T_r$-approximation for $r>1$ is not needed in this paper, but we include its definition for the sake of completeness.

\begin{defn}[Strong $T_r$-approximation]\label{defn:strong}

Let $P\subseteq \VF^m$ be definable, let  $f=(f_1,\dots,f_n) : P\to \VF^n$ be a definable function and $r > 0$ be an integer.

\begin{enumerate}
\item We say that $f$ satisfies strong $T_r$-approximation if $P$ is an open cell around zero, $f$ satisfies condition $(*)$ on $P$ and, for each model $L$ of $\cT$, the function $f_L$ satisfies $T_r$-approximation and moreover for each box $B\subseteq P(L)$, the $\cL^*$-term associated to $f$ satisfies $T_r$-approximation on $B_\as$.

\item A family $f_i : P_i \to X$  for $i\in I$ of definable functions is called a (strong) $T_r$-parametrization of $X\subseteq \VF^n$ if each $f_i$ is a (strong) $T_r$-approximation and
\[
\bigcup_{i\in I} f_i(P_i)=X.
\]
\end{enumerate}
\end{defn}

The fact that $P$ is an open cell around zero in Definition \ref{defn:strong} is particularly handy since it enables an easy description of the maximal boxes contained in $P$ which combines well with Condition $(*)$ and for composing with power maps. Global analyticity in complete models as given in Section \ref{sec:analyci} together with a calculation on the coefficients of the occurring power series will then complete the proofs of the parmeterization Theorems \ref{uniform-Tr-approx} and \ref{uniform-Tr-approx-char0}.

\begin{thm}[Strong $T_1$-parametrization]\label{strong-T1-approx}
Let $n\geq 0$, $m\geq 0$ be integers and let $X=(X_y)_{y\in Y}$ be a definable family of definable subsets $X_y \subseteq \cO_\VF^n$ for $y$ running over a definable set $Y$. Suppose that $X_y$ has dimension $m$ for each $y\in Y$. Then there exist a finite set $I$ and a definable family $g=(g_{y,i})_{(y,i)\in Y\times I}$ of definable functions
$$
g_{y,i}:P_{y,i}   \to X_y
$$
such that $P_{y,i} \subseteq \cO_\VF^{m}$ and for each $y$, $(g_{y,i})_{i\in I}$ forms a strong $T_1$-parametrization of $X_y$.
\end{thm}

\begin{proof}
We work by induction on $m$. We will repeatedly throw away pieces of lower dimension and treat them by induction. We will work uniformly in $y$.
We will also successively consider finite definable partitions of $X$ without renaming.
By Lemma \ref{lem:projection-full-dim}, up to taking a finite definable partition of $X$, we can find a locally 1-Lipschitz surjective function $f_y : P_y\subseteq \VF^m \to X_y$, with $P_y$ open for each $y\in Y$. By Theorem \ref{Lip}, we can further assume that $f_y$ is globally 1-Lipschitz on $P_y$, or equivalently, that $f_y$ satisfies $T_1$-approximation on $P_y$. By Lemma \ref{terms} we may moreover suppose that the component functions of $f$ are given by $\cL^*$-terms.
 We still need to improve $f$ and $P$ in order to have that the $f_y$ satisfy strong $T_1$-approximation, in particular, condition (*), $T_1$-approximation holds on associated boxes of boxes in its domain, and, that $P_y$ is an open cell around zero.

First we ensure, as an auxiliary step,  that the first partial derivatives of the $f_y$ are bounded by 1 on the associated box of any box in its domain $P_y$, by passing to an algebraic closure $\VF^\alg$ of $\VF$ with the natural $\cL$ and $\cL^*$ structures.  This passage to $\VF^\alg$ preserves well properties of  quantifier free formulas and of terms by results from both \cite{cluckers_fields_2011} and \cite{CLips} for the involved analytic structures on $\VF$ and on $\VF^\alg$. This step is  done by switching again the order of coordinates as in the proof of Lemma \ref{lem:projection-full-dim} where necessary. Since it is completely similar to the corresponding part of the proof of \cite[Theorem 3.1.3]{CCL-PW}, we skip the details.

Finally we show that we can ensure all remaining properties, using induction.
Apply Lemma \ref{lem-conditionstar}, uniformly in $y$, to obtain a partition of $P = (P_y)_y$ into open cells $A = (A_y)_y$ over $Y$ with center $(c_i)_{i=1}^m$ and an associated bijection $\theta_A$ in the notation of Definition \ref{cell0}, while neglecting a definable subset $B$ of $P$ where $B_y$ is of dimension less than $m$.
By induction on $m$, we may apply Theorem \ref{strong-T1-approx} (for the value $m-1$) to the graph of $(c_i)_{i=1}^{m}$ to find a strong $T_1$-parametrization for this graph.
One obtains the required parametrization of $X$ by composing the parametrization of the graph of $(c_i)_{i=1}^m$ with $\theta_A^{-1}$ and $f$. Indeed, firstly one concludes as in the proof of Lemma \ref{lem-conditionstar} that property $(*)$ is satisfied for this composition and that the domain is an open cell around zero. Secondly, the composition of 1-Lipschitz functions is 1-Lipschitz, and, the first order partial derivatives are bounded by $1$ on associated boxes of its domain.
Finally, the condition of $T_1$-approximation on each associated box follows from Proposition \ref{prop-condition-star-analyticity} and \cite[Corollary 3.2.12]{CCL-PW}, since the derivative is bounded by 1 on associated boxes of its domain. 
\end{proof}

The whole purpose of requiring the domains of strong
$T_1$-parametrizations to be cells around zero is to deduce existence of $T_r$-parametrizations from strong $T_1$-parametrizations by precomposing with power functions. This is enabled by the next two lemmas.

%


\begin{lem}
\label{lem-composition-rth-powers}
Let $f$ be a definable function on $X\subset \VF$ satisfying strong $T_1$-approximation. Then there is some $M>0$ such that for $L$ either a model of $\cT$ which is a complete discretely valued field, or, a local field with residue field cardinality at least $M$, the following holds for any integer $r>0$ and with $p_r$ be the $r$-power map, sending $x$ in $L$ to $x^r$.
For any ball open ball $B= b(1+\mathcal{M}_L)\subseteq L$ with $B\subset X$, and for any ball $D\subseteq L$ satisfying $p_r(D)\subseteq B$, the function
$$
f_r:=f_L\circ p_r
$$
satisfies $T_r$-approximation on $D$. Moreover, $f_r$ can be developed around any point $b'\in D$ as a power series  which is converging on $D_{\rm as}$ and whose coefficients $c_i$ satisfy
\[
\abs{c_i}\leq \abs{b'}^{r-i} \mbox{ for all } i>0.
\]
\end{lem}

\begin{proof}
Observe first that since the choice of $b\in B$ is arbitrary, it suffices to show the lemma for $b'\in D$ with $b'^r=b$.
Since $f$ satisfies condition $(*)$, there is a converging power series $\sum_{i\in \NN} a_i (x-b)^i$ as given by Proposition \ref{prop-condition-star-analyticity}.
Since $x\mapsto \sum_{i\in \NN} a_i (x-b)^i$ satisfies $T_1$-approximation on $B_\as$, we have
\[\abs{\sum_{i\geq 1} a_i (x-b)^i}< \abs{b}
\]
 for all $x\in B_\as$. By the relation between the Gauss norm and the supremum norm on $B_\as$, we then have
\begin{equation}
\label{eqn_bound_ai}
\abs{a_i}\leq \abs{b}^{1-i}
\end{equation}
 for all $i\geq 1$.
 Fix $b'\in D$ with $b'^r=b$. Since $f$ is given by a power series on $B$, by composition we can develop  $f_r=\sum_{k\geq 0} c_k (x-b')^k$ as a power series around $b'$.
Using multinomial development, we find that for $k\geq 1$,
\[
\abs{c_k}\leq \max_{i\geq 1}\set{ \abs{a_i}\cdot \abs{b'}^{ri-k}}.
\]
Note that we could also get an explicit expression for $c_k$ using the chain rule for Hasse derivatives.

Combining with Equation (\ref{eqn_bound_ai}) yields
\[
\abs{c_k}\leq \abs{b'}^{r-k}.
\]
In particular, we have $\abs{c_k}\leq 1$ for $k\leq r$ and
for any $x\in D$,
\[
\abs{f_r(x)-T_{f_r,b'}^{<r}(x)}=\abs{\sum_{k\geq r} c_k(x-b')^k}\leq \abs{x-b'}^r
\]
which concludes the proof.\end{proof}

We now formulate a multidimensional version of Lemma \ref{lem-composition-rth-powers}. To do so we introduce the following notations. For a tuple $i=(i_1,\dots,i_m)\in \NN^m$ and $x=(x_1,\dots, x_m)\in L^m$, recall that $x^i$ is $\prod_{1\leq k\leq m}x_k^{i_k}$ and $\abs{i}=i_1+\dots+i_m$.  Also define $\abs{x}_{\min, i}$ to be
\[
\min_{1\leq j\leq m,\ i_j>0}\set{\abs{x_j}}.
\]

The idea is also to precompose with the $r$-th power to achieve the $T_r$-property on boxes. A naive approach to estimate the coefficients of the composite function, using the maximum modulus principle on the associated box, would lead to a bound for the $i\in \NN^m$ coefficient of $\abs{b}^r\abs{b^i}^{-1}$. This however is not optimal and not enough for our needs. By working one variable at a time, we will improve it.

\begin{lem}
\label{lem-composition-rth-powers-multdim}
Let $f$ be a definable function on $X\subset \VF^m$ satisfying strong $T_1$-approximation. Then there is some $M>0$ such that for $L$ either a model of $\cT$ which is a complete discretely valued field, or, a local field with residue field cardinality at least $M$, the following holds for any integer $r>0$.

Let $b=(b_1,\dots,b_m)$ be in $L^m$ and suppose that  $B=\prod_i b_i(1+\mathcal{M}_L)\subseteq L^m$ is a subset of $X(L)$.
For any $d=(d_1,\dots,d_m)$ in $L^m$, write $p_{r,d}$ for the function $(x_1,\dots,x_m)\mapsto (d_1x_1^r,\dots,d_mx_m^r)$.
Then for any box $D\subseteq L^m$ such that $p_{r,d}(D)\subseteq B$, the function
$$
f_{r,d}:=f_L\circ p_{r,d}
$$
satisfies $T_r$-approximation on $D$.
Moreover, $f_{r,d}$ can be developed  around any point $b'\in D$ as a power series converging on $D_{\rm as}$ with coefficients $c_k$ satisfying
\[
\abs{c_k}\leq \abs{b'}_{\min,k}^r\abs{b'^k}^{-1} \mbox{ for all } k \in \NN^m\backslash \set{0}.
\]
\end{lem}

\begin{proof}
Up to rescaling, we can assume $d_1=\dots=d_m=1$. As in the proof of Lemma \ref{lem-composition-rth-powers}, we can fix $b\in B$, $b'\in D$ such that $b'^r=b$ and develop $f$ as a power series $\sum_{i\in \NN^m}a_i (x-b)^i$ that converges on $B_\as$. Fix $\hat x_1\in \hat b(1+\cM_L)^{m-1}_\as$ and consider the function
\[
f_{\hat x_1} : \begin{cases} b_1(1+\cM_L)_\as  \to  L \\
		   x_1  \mapsto f(x_1,\hat x_1).
		  \end{cases}
\]
It is given by a power series $\sum_{i_1\in \NN}a_{i_1}(\hat x_1)(x_1-b_1)^{i_1}$ around $b_1$ that converges on $b_1(1+\cM_L)_\as$.

By the $T_1$-property for $f$ on $B_\as$, we have that for any $x_1\in b_1(1+\cM_L)_\as$,
\[
\abs{f_{\hat x_1}(x_1)-f_{\hat x_1}(b_1)}=\abs{f(x_1,\hat x_1)-f(b_1,\hat x_1)}\leq \abs{x_1-b_1}\leq \abs{b_1}.
\]
Hence by the relation between the Gauss norm and the supremum norm on $b_1(1+\cM_L)_\as$, for each $i_1>0$ we have
\[
\abs{a_{i_1}(\hat x_1)}\leq \abs{b_1}^{1-i_1}.
\]
Now view $a_{i_1}(\hat x_1)$ as a function of $\hat x_1\in \hat b(1+\cM_L)^{m-1}_\as$, and by using again the relation between Gauss norm and sup norm, we find that for each $i\in \NN$ such that $i_1>0$,
\[
\abs{a_i}\leq \abs{b_1}^{1-i_1}\cdot \abs{\hat b^{(i_2,\dots,i_m)}}^{-1}=\abs{b_1}\abs{b^{i}}^{-1}.
\]
By switching the numbering of the coordinates, we get that for each $i\in \NN^m\backslash\set{0}$,
\[
\abs{a_i}\leq \abs{b}_{\min, i} \abs{b^{i}}^{-1}.
\]

The end of the proof is now similar to that of Lemma \ref{lem-composition-rth-powers}. Indeed, we develop $f_{r,d}=f\circ p_{r,d}$ into a power series around $b'$, denoted by $\sum_{c_k\in \NN^m}c_k (x-b')$. Then by multinomial development and using the bound for $a_i$ we find that for $k\in \NN^m\backslash\set{0}$,
\[
\abs{c_k}\leq  \abs{b'}_{\min,k}^r\abs{b'^k}^{-1}.
\]
It is now a direct consequence of this bound that
$\abs{c_k}\leq 1$ for $k\in \NN^m\backslash\set{0}$ with $\abs{k}<r$.

Now fix $x\in D$ and $k\in \NN^m\backslash\set{0}$ with $\abs{k}\geq r$. Choose some $\underline{r}\in \NN^m$ such that $\abs{\underline{r}}=r$ and $\underline{r}_j\leq k_j$ for $j=1,\dots,m$. We have :
\begin{align*}
\abs{c_k(x-b')^k}&\leq  \abs{b'}_{\min,k}^r\abs{b'^k}^{-1}\abs{(x-b')^k}\\
&\leq  \abs{b'}_{\min,k}^r\abs{b'^k}^{-1}\abs{(x-b')^{k-\underline{r}}}\abs{x-b'}^r\\
&\leq \abs{b'^{\underline{r}}}\abs{b'^k}^{-1}\abs{(x-b')^{k-\underline{r}}}\abs{x-b'}^r\\
&\leq \abs{x-b'}^r.
\end{align*}
Hence $f_{r,d}$ satisfies $T_r$-approximation on $D$.
\end{proof}

\begin{proof}[Proof of Theorem \ref{uniform-Tr-approx}]
First apply Theorem \ref{strong-T1-approx} to $X$. We get a finite set $I$ and a family $g=(g_{y,i})_{(y,i)\in Y\times I}$ of definable functions
$$
g_{y,i}:P_{y,i}   \to X_y
$$
such that $P_{y,i} \subseteq \cO_\VF^{m}$ and for each $y$, $(g_{y,i})_{i\in I}$ forms a strong $T_1$-parametrization of $X_y$.

By Proposition \ref{prop-condition-star-analyticity}, we find $M\in \NN$ such that for any $L\in \cC_{\cO,M}$, any  $y\in Y(L)$, any box $B \subseteq P_{y,i}(L)$ and any $b\in B$, there is a power series centered at $b$, converging on $B_\as$ and equal on $B_\as$ to $g_\as$. Fix such an $L$ and write $q$ for $q_L$.



Observe that it is enough to prove the theorem for $r$ prime to $q$. Indeed, a $T_{r+1}$-parametrization is also a $T_r$-parametrization. Hence up to enlarging the constant, if $r$ is not prime to $q$ one can apply the theorem with $r+1$ to obtain a $T_r$-parametrization.

We fix an integer $r$ prime to $q$ and we partition $\FF_q^\times$ into $\ell=\gcd(r,q-1)$ sets $A_1,\dots A_\ell$ such that $x\mapsto x^r$ is a bijection from each $A_i$ to $\left(\FF_q^{\times}\right)^r$, the set of $r$-th powers in $\FF_q^{\times}$. We choose representatives $\bar d_1,\dots, \bar d_\ell$ for cosets of $\left(\FF_q^{\times}\right)^r$ and we fix lifts of them, denoted by $d_1,\dots, d_\ell\in \cO_L$.
For $x\in \cO_L\backslash \set{0}$,  we set
$\xi (x) = d_i$ for $i$ such that $\ac(x)\in A_i$.

Now define for $j=(j_1,\dots,j_m)\in \set{0,\dots r-1}^m$ the function
\[
p_{r,j} :
\begin{cases}(\cO_L\backslash \set{0})^m \to (\cO_L\backslash \set{0})^m\\
x=(x_1,\dots,x_m) \mapsto (t^{j_1}\xi(x_1)x_1^r,\dots, t^{j_m}\xi(x_m)x_m^r),
\end{cases}
\]
where $t$ is our constant symbol for a unifomizer of $\cO_L$.

Let $D_{y,i,j}= p_{r,j}^{-1}(P_{y,i}(L))$. By compactness and up to making $M$ larger if necessary,  we have that $P_{y,i}(L)$ is a cell around zero. By Hensel's lemma, the union over $j\in \set{0,\dots r-1}^m$ of the sets $p_{r,j}(D_{y,i,j})$ is equal to $P_{y,i}(L)$. We claim that the family $(\bar g_{y,i,j}=g_{y,i}\circ p_{r,j})_{(y,i,j)\in Y(L)\times I\times \set{0,\dots r-1}^m}$ is the desired $T_r$-parametrization of $X(L)$. To lighten  notations,  let us skip for the rest of the proof the subscript $(y,i,j)$. By Lemma \ref{lem-composition-rth-powers-multdim} and up to making $M$ larger if necessary, $\bar g$ satisfies $T_r$-approximation on each box contained in $D$. We will show using $T_1$-approximation for $g$ and ultrametric computations that $\bar g$ satisfies $T_r$-approximation on the whole $D$.

Fix $x,y\in D$. If $x$ and $y$ are in the same box contained in $D$, then we are done. Assume then that they are not.

 Choose $v\in D$ such that $\ac(v_i)=\ac(y_i)$ and $\abs{v_i}=\abs{x_i}$, and in the case we moreover have $\ac(x_i)=\ac(y_i)$, set $v_i=x_i$. Such a $v$ exists by Hensel's lemma and the fact that $D$ is a cell around zero. Define $w\in D$ such that $w_i=v_i$ if $\abs{v_i}=\abs{y_i}$ and $w_i=y_i$ if $\abs{v_i}\neq\abs{y_i}$. We have that $w$ and $y$ lie in the same box contained in $D$. There are also $d,d',d''\in \cO_L^m$ as prescribed by $p_{r,j}$ such that $\bar g (x)=g(dx^r)$, $\bar g (w)=g(d'w^r)$, $\bar g (y)=g(d''y^r)$.

We then have $\abs{\bar g(x)-T_{\bar g,y}^{<r}(x)}$
\begin{align*}
&\leq  \max\set{ \abs{\bar g(x)-\bar g(w)},\abs{\bar g(w)-T_{\bar g,y}^{<r}(w)},\abs{T_{\bar g,y}^{<r}(w)-T_{\bar g,y}^{<r}(x)}}\\
&= \max\set{ \abs{g(dx^r)- g(d'w^r)},\abs{\bar g(w)-T_{\bar g,y}^{<r}(w)},\abs{T_{\bar g,y}^{<r}(w)-T_{\bar g,y}^{<r}(x)}}\\
&\leq \max\set{ \abs{dx^r-d'w^r},\abs{w-y}^r,\abs{T_{\bar g,y}^{<r}(w)-T_{\bar g,y}^{<r}(x)}}\\
&\leq \max\set{ \abs{x-y}^r,\abs{w-y}^r,\abs{T_{\bar g,y}^{<r}(w)-T_{\bar g,y}^{<r}(x)}}\\
&\leq \max\set{ \abs{x-y}^r,\abs{T_{\bar g,y}^{<r}(w)-T_{\bar g,y}^{<r}(x)}}\\
&\leq \abs{x-y}^r.
\end{align*}

The first inequality is by ultrametric triangular inequality, the second is by global $T_1$-property for $g$ and $T_r$-property on boxes for $\bar g$.  The third one is because for each $i$, we have $\abs{d_ix_i^r-d_i'w^r}\leq \abs{x_i-y_i}^r$. Indeed, there are three cases to consider. Either we have $x_i=w_i$ and $d_i=d_i'$ and then $d_ix_i^r-d_i'w^r=0$. Or we have $\abs{x_i}\neq \abs{y_i}$. In that case, $\abs{w_i}=\abs{y_i}$ and $\abs{d_i}=\abs{d'_1}\leq 1$. Then by ultrametric property we have $\abs{x_i-y_i}=\max\set{\abs{x_i},\abs{y_i}}$ and $\abs{d_ix_i^r-d_i'w^r}=\max\set{\abs{d_ix_i^r},\abs{d_i' w_i^r}}\leq \max\set{\abs{x_i},\abs{w_i}}^r=\max\set{\abs{x_i},\abs{y_i}}^r$.
The last case is when $\abs{x_i}= \abs{y_i}$ and $\ac(x_i)\neq \ac(y_i)$. In that case, $\abs{w_i}=\abs{x_i}$, $\ac(w_i)=\ac(y_i)$, $\abs{d_i}=\abs{d_i'}\leq 1$. We then have $\abs{x_i-y_i}=\abs{x_i}$ and by the choice made in the definition of $p_{r,j}$, $\ac(d_ix^r)\neq \ac(d_i'w^r)$ hence $\abs{d_ix_i^r-d_i'w^r}=\abs{d_ix^r}\leq \abs{x_i}^r=\abs{x_i-y_i}^r$.

The fourth inequality holds  because by definition of $w$, either $w_i=y_i$, either $w_i=x_i$, either $\abs{w_i}=\abs{x_i}=\abs{y_i}$ and $\ac(x_i)\neq \ac(w_i)=\ac(y_i)$. In those three cases, we have $\abs{w_i-y_i}\leq \abs{x_i-y_i}$.

To conclude the proof, it remains to prove the last inequality
\[
\abs{T_{\bar g,y}^{<r}(w)-T_{\bar g,y}^{<r}(x)}\leq \abs{x-y}^r.
\]

Suppose $T_{\bar g,y}^{<r}(x)=\sum_{k\in\NN^m,\abs{k}<r}c_k(x-y)^k$. For $A\subseteq \NN^m$, introduce the notation

\[
T_{\bar g,y}^{<r,A}(x)=\sum_{k\in\NN^m,\abs{k}<r, k\in A}c_k(x-y)^k.
\]
 Then set
 \[
 A'=\set{k=(k_1,\dots,k_m)\in \NN^m\mid k_i=0 \;\mathrm{if}\; \abs{y_i}\leq \abs{x_i-y_i}},
 \]
 and let $A$ be its complement. The condition can be rephrased in writing that $k_i=0$ if $w_i\neq x_i$. In particular, for $k\in A'$ we have $(x-y)^k=(w-y)^k$, hence $T_{\bar g,y}^{<r,A'}(x)=T_{\bar g,y}^{<r,A'}(w)$.

 Hence it remains to show that
 \[
\abs{T_{\bar g,y}^{<r,A}(w)-T_{\bar g,y}^{<r,A}(x)}\leq \abs{x-y}^r.
\]
We claim that
\[
\abs{T_{\bar g,y}^{<r,A}(x)}\leq \abs{x-y}^r \;\mathrm{and}\; \abs{T_{\bar g,y}^{<r,A}(w)}\leq \abs{x-y}^r,
\]
which implies the preceding inequality.

Since for each $i$, $\abs{w_i-y_i}\leq \abs{x_i-y_i}$, it is enough to prove that for each $k\in A$ such that $0<\abs{k}<r$,
\[
\abs{c_k(x-y)^k}\leq \abs{x-y}^r.
\]

From the definition of $A$, there is some $i_0$ such that $k_{i_0}>0$ and $\abs{y_{i_0}}\leq \abs{x_{i_0}-y_{i_0}}$. Suppose to lighten the notations that $i_0=1$. Set $\underline{r}=(\underline{r}_1,\dots,\underline{r}_m)$ with $\underline{r}_i=k_i$ for $i>1$ and $\underline{r}_i=r-\abs{k}+k_1\geq 1$.

Recall the bound for $\abs{c_k}$ obtained from Lemma \ref{lem-composition-rth-powers-multdim}.
We now compute, using this bound and the definition of $\underline{r}$ :
\begin{align*}
\abs{c_k(x-y)^k}&\leq \abs{y}_{\min,k}^r\abs{y^k}^{-1}\abs{(x-y)^k}\\
&\leq \abs{y^{\underline{r}}}\abs{y^k}^{-1}\abs{(x-y)}^\abs{k}\\
&\leq \abs{y^{\underline{r}-k}}\abs{(x-y)}^\abs{k}\\
&= \abs{y_1}^{r-\abs{k}}\abs{(x-y)}^\abs{k}\\
&\leq \abs{x_1-y_1}^{r-\abs{k}}\abs{(x-y)}^\abs{k}\\
&\leq \abs{(x-y)}^\abs{r}.
\end{align*}
This finishes the proof of the theorem. \end{proof}

\begin{proof}[Proof of Theorem \ref{uniform-Tr-approx-char0}]
The proof is  similar to the one of Theorem \ref{uniform-Tr-approx} above, using Theorem \ref{strong-T1-approx} and then precomposition by power functions. One just needs to delete the application of compactness, and, instead of using the map $\xi$ which chooses and exploits the lifts of cosets of $r$-th powers in the residue field, one uses parameters from $R_r$ to paste pieces together. (That a factor $b_r$ comes in is because in this general case the pasting is more rough by the lack of equality between the number of cosets of the $r$-th powers in the residue field equals the number of solutions of $x^r=1$ in the residue field, in general.) The rest of the proof is completely similar.
\end{proof}

\section{Points of bounded degree in $\FF_q[t]$}\label{sec:Fqt-count}

\subsection{A counting theorem}
The goal of this section is to prove the following theorem, of which Theorem \ref{thmintro_points-curves} is a particular case.
Recall from the introduction that, for $q$ a prime power and  $n$ a positive integer, $\FF_q[t]_n$ is the set of polynomials with coefficients in $\FF_q$ and degree (strictly) less than $n$, and, for an affine variety $X$ defined over a subring of  $\FF_q\llp t\rrp$,  $X(\FF_q[t])_n$ denotes the subset of $X(\FF_q\llp t\rrp)$ consisting of points whose coordinates lie in $\FF_q[t]_n$. Also, for a subset $A$ of $\FF_q\llp t\rrp^m$, write $A_n$ for the subset of $A$ consisting of points whose coordinates lie in $\FF_q[t]_n$.

For an affine (reduced) variety $X\subset \AA_R^m$ with $R$ an integral domain contained in an algebraically closed field $K$, we define the degree of $X$ as  the degree of the closure of $X_K$ in $\PP_K^m$.
For example, if $X$ is a hypersurface given by one (reduced) equation $f$, then the degree of $X$ equals the (total) degree of $f$.

\begin{thm}
\label{thm:main:uniform:countingFq}
Let $d$, $m$ and $\delta$ be positive integers. 
Then there exist real numbers $C=C(d,m,\delta)$ and $N=N(d,m,\delta)$ such that for each prime $p>N$, each power $q=p^\alpha$ with $\alpha>0$ an integer, each integer $n>0$ and each irreducible variety $X\subseteq \AA^m_{\FF_q\llp t\rrp}$ of degree $\delta$ and dimension $d$ one has
\[
\# X(\FF_q[t])_n \leq C n^2 q^{n(d-1) + \ceil{\frac{n}{\delta}}}.
\]
\end{thm}

We first give a bound for a so-called naive degree. Define the naive degree of a variety $X\subset \AA_R^m$ with $R$ an integral domain as the minimum, taken over all tuples of (nonzero) polynomials $f=(f_1,\dots,f_s)$ over $R$ with $X(K) =\{x\in K^m\mid f(x)=0\}$,  of the product of the degrees of the $f_i$.

\begin{lem}
\label{lem:naive:deg}
Let $d$, $m$, and $\delta$ be positive integers. 
Then there exist numbers $C=C(d,m,\delta)$ and $N=N(d,m,\delta)$ such that for each prime $p>N$, each power $q=p^\alpha$ with $\alpha>0$ an integer, and each geometrically irreducible variety $X\subseteq \AA^m_{\FF_q\llp t\rrp}$ of degree $\delta$ and dimension $d$,
one has that the naive degree of $X$ is bounded by $C$.
\end{lem}
\begin{proof}
From the theory of Chow forms, see \cite{sam_alg_ga} or \cite{cat_chow}, a variety  $X\subseteq \AA^m_{\FF_q\llp t\rrp}$ of degree $\delta$ and dimension $d$ is determined set-theoretically by a hypersurface of degree $\delta$ in the Grasmanniann of $G(m-d-1,m)$ of $m-d-1$-dimensional vector subspaces of the $m$-dimensional space. As explained for example in \cite{cat_chow}, one can construct from such a hypersurface a system of $m(d+1)$ equations of degrees at most $\delta$ such that their zero set coincide set-theoretically with $X$. Hence the naive degree of $X$ is bounded by $\delta m(d+1)$.
\end{proof}

The following  trivial bound for points of bounded height is typical.

\begin{lem}
\label{lem:triv:countingFq}
Let $d$, $m$,  and $\delta$  be positive integers. 
Then there exist real numbers $C=C(d,m,\delta)$ and $N=N(d,m,\delta)$ such that for each prime $p>N$, each power $q=p^\alpha$ with $\alpha>0$ an integer, each integer $n>0$ and each irreducible variety $X\subseteq \AA^m_{\FF_q\llp t\rrp}$ of degree $\delta$ and dimension $d$,
one has
\[
\# X(\FF_q[t])_n \leq C q^{nd} .
\]
\end{lem}
\begin{proof}
The lemma follows easily from Noether's normalization lemma and Lemma \ref{lem:naive:deg}.
\end{proof}

%

Let us first  reduce  the statement of Theorem \ref{thm:main:uniform:countingFq}  to the case of planar curves, similarly as in
\cite{Pila-ast-1995}. In this section, definable means definable in the language $\cL_\mathrm{DP}$ of Setting \ref{setting-T_DP} and with $\cO=\ZZ$.

\begin{proof}[Reduction of Theorem \ref{thm:main:uniform:countingFq} to the case $m=2$ and $d=1$.]
Fix positive integers $d,m,\delta$.
By Lemma \ref{lem:naive:deg}, irreducible varieties in $\AA^m$ of dimension $d$ and of degree $\delta$ form a definable family of definable sets, say, with parameter $z$ in a definable (and Zariski-constructible) set $Z$; write $X_z$ for the variety in $\AA^m$ corresponding to the parameter $z\in Z$.
Assume first that $m>2$ and $d=1$. Consider the family of linear projections $p_{a,b} : \AA^m\to \AA^2$ written in coordinates $x=\sum a_i x_i$ and $y=\sum b_i y_i$ and with parameters $(a,b)\in \AA^{2m}$. Then, for each $z\in Z$, there is a non-empty Zariski open subset of parameters $O_z\subseteq \AA^{2m}$ such that first of all $p_{a,b}$ is surjective and secondly, the varieties $X_z$ and $p_{a,b}(X_z)$ have the same degree $\delta$ (and are both irreducible of dimension $1$) for all $(a,b)\in O_Z$. Clearly the opens $O_z$ form a definable family of definable sets with parameter $z\in Z$.

Now suppose that the prime $p$ is large enough and that $q=p^\alpha$ for some $\alpha$. Since the complement of $O_z$ is of dimension less than $2m$ by Lemma \ref{lem:triv:countingFq}, and since the $O_z$ form a definable family, we can find for each $z\in Z(\FF_q\llp t  \rrp )$ a point $(a^0,b^0)$  in $O_z(\FF_q[t])_1$ (hence, so to say, a tuple of polynomials in $t$ over $\FF_q$ and of degree $0$).
Hence, $p_{a^0,b^0}$ maps points in $\FF_q[t]^m_n$ to points in $\FF_q[t]^2_n$.
Furthermore, the fibers of $p_{a^0,b^0}$ on $X_z$ are finite, uniformly in $z$, say, bounded by $C$.
We thus have that for each large enough $p$, each $z$ in $Z(\FF_q\llp t \rrp )$, and each $n>0$, that
$$
\# X_{z}(\FF_q[t])_n  \leq C \# p(X_z) (\FF_q[t])_n.
$$
Hence the result for $d=1$ and general $m>1$ follows from the case $d=1$ and $m=2$.

Assume now that $m\geq 2$ and $d>1$. By a projection argument as above, we can assume that $d=m-1$. Consider the family of hyperplanes $H = H_{\alpha,b}$ with equation $\sum \alpha_i x_i=b$ and parameters $\alpha$ and $b$. Then for each $z\in Z$ there is a non-empty Zariski open subset $O_z$ of $\AA^{m+1}$ such that if $(\alpha,b)$ lies in  $O_z$, then $X_z\cap H_{\alpha,b}$ is irreducible, of degree $\delta$ and dimension $d$. Hence, similarly as above, for large enough primes $p$ and with $q=p^\alpha$,
we can find, for each $z$ in $Z(\FF_q\llp t \rrp )$ a point  $(\alpha^0,b^0)$  in $O_z(\FF_q[t])_1$. Now consider the family of hyperplanes $H_b$ of equations $\sum \alpha^0_i x_i =b$ with parameter $b$ running over $\FF_q\llp t \rrp$. Since $(\alpha^0,b^0)$ belongs to $O_z(\FF_q[t])_1$, and by construction, there are at most finitely many values for $b$ such that $(\alpha^0,b)\not\in O_z(\FF_q\llp t \rrp )$, say, $b_1,\dots,b_k$. In any case we can assume that $X_z\cap H_{b_j}$ is of dimension at most $m-1$ for each $j$, and hence that
$$
\# (X_z\cap H_{b_j})        \leq C q^{m-1 }
$$
for some $C$ which is independent of $q$ and $n$, by Lemma \ref{lem:triv:countingFq}.
To treat the remaining part, we apply the induction hypothesis to $X'_z=(X_z\cap H_b)$ for $b$ outside $\set{b_1,\dots,b_k}$, and we take the sum of the bounds over all values of $b$ in $\FF_q[t]_n$.
\end{proof}

\label{sec-point-bnd-degree}
\subsection{Determinant lemma}
We fix the following notations for the rest of the paper. For $\alpha=(\alpha_1,\cdots,\alpha_m)$ in $\NN^m$, set $\abs{\alpha}=\alpha_1+\cdots+\alpha_m$. Set also
\[
\Lambda_m(k):=\set{\alpha\in \NN^m\mid \abs{\alpha}=k}, \Delta_m(k):=\set{\alpha\in \NN^m\mid \abs{\alpha}\leq k},
\]
\[
L_m(k):=\# \Lambda_m(k), D_m(k)=\# \Delta_m(k).
\]

\begin{lem}[{\cite[Lemma 3.3.1]{CCL-PW}}]
\label{lem:det:estimate}
Let $K$ be a discretely valued henselian field.
Fix $\mu,r \in \NN$, and $U$ an open subset of $K^m$ contained in a box that is a product of $m$ closed balls of valuative radius $\rho$.
Fix $x_1,\ldots,x_\mu\in U$, and functions $\psi_1,\ldots,\psi_\mu : U\to K$.

Assume the following :
\begin{itemize}
\item the integer $r$ satisfies
\[
D_m(r-1)\leq \mu< D_m(r);
\]
\item
the functions $\psi_1,\ldots,\psi_\mu$ satisfy $T_r$ on $U$.
\end{itemize}
Then
\[
\ord_t(\det(\psi_i(x_j)))\geq \rho e,
\]
where $e=\sum_{i=0}^{r-1}iL_m(i)+r(\mu-D_m(r-1))$.
\end{lem}

\subsection{Hilbert functions}

Fix a field $K$.
For $s$ a positive integer, denote $K[x_0,\ldots,x_n]_s$ the space of homogenous polynomials of degree $s$. Let $I$ a homogenous ideal of $K[x_0,\ldots,x_n]$, associated to an irreducible variety of dimension $d$ and degree $\delta$ of $\PP^n_K$. Let $I_s=I\cap  K[x_0,\ldots,x_n]_s$ and $\HF_I(s)=\dim_K K[x_0,\ldots,x_n]_s/I_s$ the (projective) Hilbert function of $I$. The Hilbert polynomial $\HP_I$ of $I$ is a polynomial such that for $s$ big enough, $\HP_I(s)=\HF_I(s)$. It is a polynomial of degree $d$ and leading coefficient $\delta/d!$.

Fix some monomial ordering in the sense of \cite{cox_ideals_2015}. Denote by $\LT(I)$ the ideal generated by leading terms of elements of $I$. By \cite{cox_ideals_2015}, the Hilbert functions of $I$ and $\LT(I)$ are equal. It follows that
\[
\HF_I(s)=\# \set{ \alpha\in \Lambda_{n+1}(s)\mid x^\alpha\notin \LT(I)}.
\]
Define also for $i=0,\ldots,n$,
\begin{equation}
\label{eqn-def-sigma}
\sigma_{I,i}(s)=\sum_{\alpha\in \Lambda_{n+1}(s), x^\alpha\notin \LT(I)} \alpha_i.
\end{equation}
Hence, we have $s\HF_I(s)=\sum_{i=0}^n \sigma_{I,i}(s)$. The function $\sigma_{I,i}$ is also equal to a polynomial function of degree at most $d+1$, for $s$ large enough. It follows that there exist non-negative real numbers $a_{I,i}$ such that
\begin{equation}
\label{eqn-HPratio-DL}
\frac{\sigma_{I,i}(s)}{s\HP_I(s)}=a_{I,i}+O(1/s)
\end{equation}

when $s$ goes to $+\infty$.

We will also use the following lemma of Salberger \cite{Salberger-dgc}, which is the reason why we will use a projective embedding in the proof of Theorem \ref{thm:main:uniform:countingFq}.
\begin{lem}[\cite{Salberger-dgc}]
\label{lem-salberger}
Let $X$ be a closed equidimensional subscheme of dimension $d$ of $\PP^m_K$. Assume that no irreducible component of $X$ is contained in the hyperplane at infinity defined by $x_0=0$. Let $<$ be the monomial ordering defined by $\alpha \leq \beta$ if $\abs{\alpha}<\abs{\beta}$ or $\abs{\alpha}=\abs{\beta}$ and for some $i$, $\alpha_i>\beta_i$ and $\alpha_j=\beta_j$ for $j<i$. Then
\[
a_{I,1}+\ldots+a_{I,m}\leq \frac{d}{d+1}.
\]
\end{lem}

\subsection{Proof of Theorem \ref{thm:main:uniform:countingFq} for $m=2$ and $d=1$}

Fix a positive integer $\delta$.
Clearly all irreducible curves in $\AA^2$ of degree $\delta$ form a definable family of definable sets, say, with parameter $z$ in a definable (and Zariski-constructible) set $Z$; write $X_z$ for the curve in $\AA^2$ corresponding to the parameter $z\in Z$.


Apply Theorem \ref{uniform-Tr-approx} to the definable family of the definable sets $X_z$. 
It gives some constant $C$ and, for some $M$, for all local fields $K $ in $\cB_{\ZZ,M}$ and all integers $r>0$ prime to $q_K$, a $T_r$-parametrization of $X_z(\cO_K)$ with $Cr$ many pieces.
Fix such a $K$ and a parameter $z\in Z(K)$ corresponding to an irreducible curve $X_z\subset \AA^2_K$ of degree $\delta$.

Consider the map
\[
\iota : \begin{cases} \AA^2_K\to \AA^3_K \\ (x,y)\mapsto (1,x,y) \end{cases}
\]
and the corresponding embedding
\[
\underline\iota :  \begin{cases} \AA^2_K \longhookrightarrow
\PP^2_K \\ (x, y) \mapsto [1: x: y]. \end{cases}
\]
Denote by $I_z$ the homogenous ideal associated to the closure of $\underline\iota(X_z)$.  


Fix some positive integer $s$, set
\[
M_z(s)=\set{\alpha\in \Lambda_{3}(s), x^\alpha\notin \LT(I_z)},
\]
$\mu=\# M_z(s)$ and $e=\mu(\mu-1)/2$.

Now consider the given $T_r$-parametrization of $X_z(\cO_K)$ with  $r=\mu$ and work on one of the $C\mu$ pieces $U_z\subseteq \cO_K$ with function $g_z : U_z \to X( \cO_K )$ satisfying $T_{\mu}$ on $U_z$.

Fix a closed ball $B_\beta\subseteq \cO_K$ of valuative radius $\beta$. Fix some points $y_1,\ldots, y_\mu$ in $(g(B_\beta\cap U))_n$ and consider the determinant
\[
\Delta=\det(\iota(y_i)^\alpha)_{1\leq i\leq \mu, \alpha \in M_z(s)}.
\]

Since the composition of functions satisfying $T_{\mu}$ also satisfies $T_{\mu}$, we can apply Lemma \ref{lem:det:estimate}, with $m=1$, $r=\mu$ to get that
\[
\ord_t \Delta \geq \beta e.
\]

On the other hand, since the points $y_i$ are of degree less than $n$ as polynomials in $t$ over $\FF_{q_K}$, we also have
\[
\deg \Delta \leq (n-1)(\sigma_1+\sigma_2),
\]
where $\sigma_1,\sigma_2$ are defined by Equation (\ref{eqn-def-sigma}).
Hence, if $\Delta$ is not zero, then
\[
\ord_t \Delta \leq (n-1)(\sigma_1+\sigma_2).
\]
It follows that $\Delta=0$ whenever
\begin{equation}
\label{inequality-det}
\beta e > (n-1)(\sigma_1+\sigma_2).
\end{equation}
When such an inequality holds, the matrix $A=(y_i^\alpha)$ is of rank less than $\mu$. Fix a minor of maximal rank $B=(y_i^\alpha)_{i\in I,\alpha \in J}$ and some $\alpha_0\in M_z(s)\backslash J$.
Then the polynomial
\[
f(x,y)=\det
\left(\begin{array}{c}
y_i^\alpha \\
(1,x,y)^\alpha
\end{array}\right)_{i\in I, \alpha \in J\cup \set{\alpha_0}}
\]
is of total degree at most $s$ and nonzero, since the coefficient of $(1,x,y)^{\alpha_0}$ is $\det(B)$. Moreover, it vanishes at all points in $g(B_\beta\cap U)_n$ but does not vanish on the whole $X_z$, since its exponents lie in $M_z(s)$ and $X_z$ is irreducible. Hence by B\'ezout's theorem, there are at most $s\delta$ points in  $(g(B_\beta\cap U))_n$.

We now show how to choose $s$ and $\beta$ in terms of $n$ such that inequality (\ref{inequality-det}) holds. Recall that $\mu=\# M_z(s)=\HF_{I_z}(s)$. By properties of Hilbert polynomials and equation (\ref{eqn-HPratio-DL}), we have
\begin{equation}
\label{eqn-mu-HP}
\mu={\delta} s+O(1)
\end{equation}
and
\[
\frac{\sigma_i}{\mu}=a_is+O(1).
\]
Combining those two equations, we get
\[
\sigma_i=a_i{\delta} s^2+ O(s)
\]
and
\[
e=\frac{{\delta}^2s^2}{2} +O(s),
\]
and finally, by applying Lemma \ref{lem-salberger},
\[
\frac{\sigma_1+\sigma_2}{e}\leq \frac{1}{{\delta}}+O(s^{-1}).
\]
Hence there is some $s_0$ and $C_0>0$ such that for every $s\geq s_0$,
\[
\frac{\sigma_1+\sigma_2}{e}\leq \frac{1}{{\delta}}+C_0s^{-1}.
\]
Recall that the coefficients of Hilbert polynomials can be bounded in terms of the degree of the curve and that the characteristic is assumed to be large.
Hence $s_0$ and $C_0$ depend only on the degree $\delta$ of the curve $X_z$.

If follows that for
\begin{equation}\label{eqn-s}
s=\ceil{\max\{s_0,2C_0(n-1)\}},
\end{equation}
 we have
\[
(n-1)\frac{\sigma_1+\sigma_2}{e}\leq \ceil{\frac{n}{{\delta}}}.
\]
We can thus set $\beta=\ceil{\frac{n}{{\delta}}}$ to satisfy inequality (\ref{inequality-det}).
It follows from the preceding discussion that there are at most $s\delta$ points in $g(B_\beta\cap U)_n$.  From Equation (\ref{eqn-mu-HP}), we have $\mu\leq {\delta}s+C_1$, for some constant $C_1$, and from (\ref{eqn-s}) that $s\leq C_2n$ for some constant $C_2$, with $C_i$ independent of $n$. Since we need $q^\beta$ closed balls of valuative radius $\beta$ to cover $\FF_q \llb t\rrb=\cO_K$, and that we have a $T_\mu$-parametrization of $X(\FF_q \llb t\rrb)$ involving $C\mu$ pieces, we find that (after enlarging $C$) there are at most
$$
Cn^2 q^{\ceil{\frac{n}{{\delta}}}}
$$ 
points in $X(\FF_q[t])_n$.
\qed

\begin{remark}
In their preprint \cite{bnd_2tor}, Bhargava et al. use Sedunova's result \cite{sedunova_BP} to bound the 2-torsion of class groups of function fields over finite fields, see their Theorem 7.1. One can use instead our Theorem \ref{thm:main:uniform:countingFq} in the special case of Theorem \ref{thmintro_points-curves} to obtain a uniform version of their result. We thank Peter Nelson for directing us to the reference \cite{bnd_2tor}.
\end{remark}

\section{Uniform non-Archimedean Pila-Wilkie counting theorem}
\label{section_PiWi}

In this section we provide uniform versions in the $p$-adic fields for large $p$ and also in the fields $\FF_q\llp t \rrp$ of large characteristic
of  several of the main counting results of \cite{CCL-PW} (on rational points on $p$-adic subanalytic sets). To achieve this we use the uniform parameterization result of Theorem \ref{uniform-Tr-approx}. Furthermore, Proposition \ref{prop-covering-bhp-hypersurfaces} is new in all senses, and is a (uniform) non-archimedean variant of recent results of \cite{CPW_UYGpara}, \cite{BN3}; it should be put in contrast with Proposition 4.1.3 of \cite{CCL-PW}.



\subsection{Hypersurface coverings}

We begin with fixing some terminology.

Consider the language $\cL = \cL_\mathrm{DP}^\mathrm{an}$ as described in Setting \ref{setting-T_DP}.  We will from now on only consider definable sets which are subsets of the Cartesian powers of the valued field sort (sometimes in a concrete $\cL$-structure, and sometimes for the theory $\cT$).

\begin{defn}
Let $K$ be an $\cL$-structure.
An $\cL(K)$-definable set $X\subset K^n$ is said to be of dimension $d$ at $x\in X$ if for every small enough box containing $x$, $X\cap B$ is of dimension $d$.
An $\cL(K)$-definable set $X\subset K^n$ is said to be of pure dimension $d$ if it is of dimension $d$ at all points $x$ in $X(K)$.

For an $\cL(K)$-definable set $X\subset K^n$, define the algebraic part $X^\alg$ of $X$ to be the union of all quantifier free $\cL_\mathrm{DP}(K)$-definable sets
of pure positive dimension and contained in $X$. Note that the set $X^\alg$ is in general neither semi-algebraic nor subanalytic.
 \end{defn}

By subanalytic we mean from now on $\cL$-definable, or $\cL(K)$-definable if we are in a fixed $\cL$-structure, and we speak about definable families in the sense explained just below \ref{sec:not}. Likewise, by semi-algebraic we mean definable set in the language $\cL_\mathrm{DP}$, or $\cL_\mathrm{DP}(K)$-definable if we are in a fixed structure (see section \ref{setting-T_DP}). Write $\cT$ for $\cT_\mathrm{DP}^\mathrm{an}$.

\begin{remark}
Observe that the definition of the algebraic part is insensitive to having or not having algebraic Skolem functions on the residue field. Indeed, its definition is local and allows parameters from the structure.
\end{remark}

If $x\in \ZZ$, set $H(x)=\abs{x}$, the absolute value of $x$. If $x=(x_1,\dots,x_n)\in \ZZ^n$, set $H(x)=\max_i\set{H(x_i)}$. If $L$ is a local field of characteristic zero, $B\geq 1$ and $X\subseteq L^n$, we set
\[
X(\ZZ,B)=\set{x\in X\cap \ZZ\mid H(x)\leq B}.
\]
 If $x\in \FF_q[t]$, we set
 $$
 H(x)=q^{\deg_t(x)},
 $$
with $\deg_t(x)$ the degree in $t$ of the polynomial $x$ over $\FF_q$. For $x=(x_1,\dots,x_n)\in (\FF_q[t])^n$, put $H(x)=\max_i\set{H(x_i)}$. We now set for $X\subseteq \FF_q\llb t \rrb$ and $B\geq 1$
 \[
 X(\FF_q[t],B)=\set{x\in X\cap \FF_q[t]\mid H(x)\leq B}.
 \]

Recall the notations at the beginning of Section \ref{sec-point-bnd-degree}. For every integers $d$, $n$, $m$, set $\mu=D_n(d)$ and let $r$ be the smallest integer such that $D_m(r-1)\leq\mu <D_m(r)$. Then set $V=\sum_{k=0}^d k L_n(k)$, $e=\sum_{k=1}^{r-1}kL_m(k)+r(\mu-D_m(r-1))$.

The following result refines Lemma 4.1.2 of \cite{CCL-PW} and has a similar proof.
\begin{lem}
\label{lem-Tr-param-hypersurfaces}
For every integers $d$, $n$, $m$ with $m<n$, consider the integers $r$, $V$, $e$ as defined above. Fix a local field $L$, a subanalytic subset $U\subseteq \cO_L^m$ and subanalytic functions $\psi=(\psi_1,\dots \psi_n) : U\to \cO_L^n$ that satisfy $T_r$-approximation and $H$. Then if $L$ is of characteristic zero,  the set $\psi(U)(\ZZ,H)$ is contained into at most
\[
q^m (\mu !)^{m/e} H^{mV/e}
\]
hypersurfaces of degree at most $d$. If $L$ is of positive characteristic, the set $\psi(U)(\FF_q[t],H)$ is contained into at most
\[
q^m H^{mV/e}
\]
hypersurfaces of degree at most $d$. Moreover, when $d$ goes to infinity, $mV/e$ goes to $0$.
\end{lem}

\begin{proof}
Troughout the proof, we use the notations introduced at the beginning of Section \ref{sec-point-bnd-degree}. Under the hypothesis of the lemma, fix a closed box $B\subseteq \cO_L^m$ of valuative radius $\alpha$. Then fix points $P_1,\dots P_\mu\in \psi(B\cap U)(\ZZ,H)$ (or $\psi(B\cap U)(\FF_q[t],H)$) and consider $x_i\in B\cap U$ such that $\psi(x_i)=P_i$. Consider the determinant $\Delta=\det((\psi(x_i)^j)_{1\leq i \leq \mu,j\in\Delta_n (d)}$.
Since $\psi$ satisfies $T_r$-approximation, it follows from Lemma \ref{lem:det:estimate} that $\ord (\Delta)\geq \alpha e$.

In the positive characteristic case, since the $P_i$ are in $\FF_q[t]$, of degree less or equal than $\log_q(H)$, if $\Delta\neq 0$, then $\ord(\Delta)\leq \log_q(H)V$. Hence if $\alpha> \log_q(H) V/e$, then $\Delta=0$.

In the characteristic zero case, since the $P_i$ are in $\ZZ$ of height at most $H$, it follows that $\Delta\in \ZZ$ is of (Archimedean) absolute value at most $\mu ! H^V$. If $\Delta\neq 0$, this implies that $\ord(\Delta)\leq \log_q (\mu ! H^V)$. Hence if $\alpha> \log_q(\mu!H^V)/e$, then $\Delta=0$.

We now assume that $\alpha$ is chosen such that $\Delta=0$.  As in the Bombieri-Pila case, by considering minors of maximal rank, we can produce a hypersurface $H$ of degree $d$ such that all the $P_i$ are contained in $H$. See the proof of Theorem \ref{thm:main:uniform:countingFq} for details.

Since we need $q^{m\alpha}$ boxes of radius $\alpha$ to cover $\cO_L^m$, in the characteristic zero case, we find that we can cover $\psi(U)(\ZZ,H)$ by $q^m{\mu!}^{m/e}H^{mV/e}$ hypersurfaces of degree $d$. In the positive characteristic case, we can cover $\psi(U)(\FF_q[t],H)$ by $q^mH^{mV/e}$ hypersurfaces of degree at most $d$.

By an explicit computation, see \cite[page 212]{pilapointsdilat} for details, we have $e\sim_d C_1(m,n)d^{n+n/m}$ and $V\sim_d C_2(m,n) d^{n+1}$, the equivalents being for $d\to +\infty$. Hence since $m<n$, $mV/e$ goes to zero when $d\to +\infty$.
\end{proof}

\begin{prop}
\label{prop-covering-bhp-hypersurfaces}
Let integers $m\geq 0$ and $n>m$ be given.
Let $X=(X_y)_{y\in Y}\subseteq (\VF^n)_{y\in Y}$ be an $\cL$-definable family of subanalytic sets with $X_y$ of dimension $m$ in each model $K$ of $\cT$ and each $y$ in $Y(K)$. Then there are a constant $C(X)$ depending only on $X$, a constant $C'(n,m)$ depending only on $n$ and $m$, and an integer $N=N(X)$ such that for each $H\geq 2$ and each local field $L\in \cC_{\cO,N}$, the following holds.

For $y\in Y(L)$ and $H\geq 2$, the set $X_y(L)(\ZZ,H)$ (resp.~$X_y(L)(\FF_{q_L}[t],H)$ for the positive characteristic case) is covered by at most
\[
C(X)q_L^m\log(H)^\alpha
\] hypersurfaces of degree at most $C'(n,m)\log(H)^{m/(n-m)}$.

Moreover, we have $\alpha=\frac{nm}{(m-1)(n-m)}$ if $m>1$ and $\alpha=\frac{n}{n-1}$ if $m=1$.
\end{prop}

\begin{proof}
We work inductively on $m$. The case $m=0$ being clear, since the cardinal of the fibers is then uniformly bounded in $y$.  Assume now  that $1\leq m$. Apply the parametrization theorem \ref{uniform-Tr-approx} to the definable family $X$.

We keep notations from the proof of Lemma \ref{lem-Tr-param-hypersurfaces}. Choose $d$ in function of $H$ such that $H^{mV/e}$ is bounded (say by 2). From the computations at the end of the proof of Lemma \ref{lem-Tr-param-hypersurfaces}, we can choose $d\sim_HC'(m,n) \log(H)^{\frac{m}{n-m}}$

We have $\mu\sim_H C_3(n,m)d^{n}$, and since $r$ is the smallest integer such that $D_m(r-1)\leq \mu <D_m(r)$, we have if $m>1$, $r=O_H(\mu^{1/(m-1)})$ and if $m=1$, $r=\mu$.
From Theorem \ref{uniform-Tr-approx}, we find a $T_r$-parametrization of $X$ involving $C(X)r$ pieces. From Lemma \ref{lem-Tr-param-hypersurfaces}, the points of height at most $H$ on one of the pieces are included in at most $q_L^m (\mu !)^{m/e} H^{mV/e}$ (if $L\in \cA_\cO$) or $q_L^m H^{mV/e}$ (if $L\in \cB_\cO$) hypersurfaces of degree at most $d$. From the Stirling formula, we see that $(\mu !)^{m/e}$ is bounded. Hence overall, up to enlarging $C(X)$, we find that $X_y(L)(\ZZ,H)$ or $X_y(L)(\FF_{q_L}[t],H)$ is contained in
\[
C(X)q_L^m\log(H)^{\alpha}
\]
hypersurfaces of degree at most $C'(n,m)\log(H)^{m/(n-m)}$, with $\alpha=\frac{nm}{(m-1)(n-m)}$ if $m>1$ and $\alpha=\frac{n}{n-1}$ if $m=1$.
\end{proof}

%

\subsection{Blocks}

In this final section, we provide uniform
versions of
results of \cite[Section 4.2]{CCL-PW} for local fields of large residue characteristic, in particular of Theorems 4.2.3 and 4.2.4 of \cite{CCL-PW}. We thus obtain analogues of  Pila-Wilkie counting results, uniformly for local fields of large enough positive characteristic. We will leave proofs, which are analogous to the ones for Theorems 4.2.3 and 4.2.4 of \cite{CCL-PW}, to the reader.

\begin{defn}
A subset $W\subset K^m$, with $K$ an $\cL$-structure, is called a block if it is either a singleton, or, it is a smooth subanalytic set of pure dimension $d>0$ contained in a smooth semi-algebraic set of pure dimension $d$.

A family of blocks $W\subseteq \VF^{m+s}$, with parameters running over $\VF^s$,  is a subanalytic set $W$ such that there exists an integer $s'\geq 0$ and a semi-algebraic set $W'\subseteq \VF^{m+s'}$ such that for each model $K$ of $\cT$, for each $y\in K^s$ there is an $y'\in K^{s'}$ such that both $W_y(K)$ and $W'_{y'}(K)$ are smooth of the same pure dimension and such that $W_y(K)\subseteq W'_{y'}(K)$.
\end{defn}

Note that if $W$ is a block of positive dimension, then $W=W^\alg$. 

Note that our notion of family of blocks, which corresponds to the one in \cite{CLL_NA_AxLind}, is a strengthening of the one in \cite{CCL-PW}  which solely ask that a family of blocks $W$ is such that $W_y$ is a block for each $y\in Y$. However, all the results in Section 4.2 of \cite{CCL-PW} hold with this strengthened definition.

Let $L$ be in $\cA_{\cO}$ and let $k>0$ be an integer. We define the $k$-height of $x\in L$ as
\[
H_k(x)=\min_a\set{H(a)\mid a=(a_1,\dots,a_k)\in \ZZ^k, \sum_{i=0}^k a_i x^i=0, a\not=0}
\]
and for $x=(x_1,\dots,x_n)\in L^n, H_k(x)=\max_i\set{H(x_i)}$.

Let $L\in \cB_{\cO}$ and $k>0$ be an integer. We define the $k$-height of $x\in L$ as
\[
H_k(x)=\min_a\set{H(a)\mid a=(a_1,\dots,a_k)\in \FF_{q_L}[t]^k, \sum_{i=0}^k a_i x^i=0, a\not=0}
\]
and for $x=(x_1,\dots,x_n)\in L^n, H_k(x)=\max_i\set{H(x_i)}$.

If $X\subseteq L^n$, we set
$$
X(k,H)=\set{x\in X\mid H_k(x)\leq H}.
$$

The following result is a generalized and  uniform version of Theorems 4.2.3 and 4.2.4 of \cite{CCL-PW}.

\begin{thm}\label{thm:unifCCL-PW}
Let $X=(X_y)_{y\in Y}\subseteq (K^n)_{y\in Y}$ be a subanalytic family of subanalytic sets of dimension $m<n$ in each model of $\cT$. Fix $\varepsilon>0$. Then there are a positive constant $C(X,k,\varepsilon)$, integers $l=l(X,k,\varepsilon)$, $N=N(X,k,\varepsilon)$, $\alpha=\alpha(m,n,k)$, and a family of blocks $W=(W_{(y,s})_{(y,s)\in Y\times K^l}\subseteq K^n\times Y\times K^l$  such that the following holds.

For each $L\in \cC_{\cO, N}$, $H\geq 1$ and $y\in Y(L)$, there is a subset $S=S(X,k,L,H,y)\subseteq K^s$ of cardinal at most $C(X,\varepsilon)q^\alpha H^\varepsilon$ such that
\[
X_y(L)(k,H)\subseteq \bigcup_{s\in S} W_{y,s}.
\]
In particular, if we denote by $W_y^\varepsilon$ the union over $s\in S$ of the $W_{y,s}(L)$ of positive dimension, we have $W_y^\varepsilon\subseteq X_y(L)^\alg$ and
\[
\# (X_y(L)\backslash W_y^\varepsilon)(k,H)\leq C(X,\varepsilon)q^\alpha H^\varepsilon.
\]
\end{thm}

The proof of Theorem \ref{thm:unifCCL-PW} is completely similar to those of \cite[Section 4.2]{CCL-PW} (namely to the proofs of Proposition 4.2.2 and Theorems 4.2.3 and 4.2.4), where instead of using \cite[Proposition 4.2]{CCL-PW}, one uses Proposition \ref{prop-covering-bhp-hypersurfaces}. We skip the proofs and refer to \cite{CCL-PW} for details.

\begin{remark}
Note also that the bound in Proposition \ref{prop-covering-bhp-hypersurfaces} is polylogarithmic, whereas the bound of \cite[Proposition 4.2]{CCL-PW} is subpolynomial. However, this improvement does not guarantee a polylogarithmic bound in the counting theorems. As in the o-minimal case, such a bound is not expected to hold in general, but might be true in some specific situations, similar to the context of Wilkie's conjecture for $\RR^{\rm exp}$-definable sets.
\end{remark}

\bibliographystyle{abbrv}
\bibliography{bibliography_YGP}

\end{document}